%% file: classification.of.simple.quartics.up.to.equisingular.deformation.tex
\documentclass{amsart}
\usepackage{amssymb,amsmath}
\usepackage{faktor}

\swapnumbers
\newtheorem{theorem}[subsubsection]{Theorem}
\newtheorem{proposition}[subsubsection]{Proposition}
\newtheorem{corollary}[subsubsection]{Corollary}
\newtheorem{lemma}[subsubsection]{Lemma}
\theoremstyle{remark}
\newtheorem{remark}[subsubsection]{Remark}
\numberwithin{equation}{section}
\theoremstyle{definition}
\newtheorem{definition}[subsubsection]{Definition}

\def\eg{\emph{e.g.}}

{\catcode`\@=11 \gdef\mnote#1{\marginpar{\footnotesize
 \tolerance\@M\spaceskip2.6\p@ plus10\p@ minus.9\p@\rm#1}}}
\def\Dg:{\endgraf{\bf Dg:\enspace}\ignorespaces}

\let\bold\mathbf

\newcommand{\be}{\begin{equation}}
\newcommand{\ee}{\end{equation}}

\let\curve=B  

\let\1e 
\def\rel(#1):#2\endrel{$$
 \hbox to\displaywidth{$(#1):$\hfill$\displaystyle#2$\hfill}$$}

\makeatletter
\def\sseat@@#1\endss@@{}
\def\sssub@{_{\the\count@}\ssnext@}
\def\ssother@#1{\def\ss@@{#1\ssnext@}}
\def\ssnext@#1{\ifcat
 A\noexpand#1\def\ss@@{\bold#1\afterassignment\sssub@\count@}\else
 \ifx +#1\oplus\let\ss@@\ssnext@\else
 \ifx .#1\let\ss@@\sseat@@\else
 \ifcat _\noexpand#1\def\ss@@##1{_{##1}\ssnext@}\else
 \ifcat ^\noexpand#1\def\ss@@##1{^{##1}\ssnext@}\else
 \ssother@{#1}\fi\fi\fi\fi\fi\ss@@}
\def\singset@#1\endss@@{\edef\ss@@{#1}\expandafter\ssnext@\ss@@.\endss@@}
\def\singset#1{\ifmmode\singset@#1\endss@@\else$\singset@#1\endss@@$\fi}
\def\SSet#1\endSSet{\singset{#1}}
\makeatother

\def\cN{\mathcal{N}}

\def\s{\mathbf{S}}

\def\disc{\operatorname{disc}}

\def\Br{\operatorname{Br}}
\def\Aut{\operatorname{Aut}}
\def\det{\operatorname{det}}

\let\PLUS+
{\catcode`\*\active \catcode`\[\active \catcode`\+\active
\let\+\relax
}

\begin{document}

\title{Classification of Simple Quartics up to Equisingular Deformation }

\author[\c{C}.~G\"{u}ne\c{s} ~Akta\c{s} ]{\c{C}\.{i}sem G\"{u}ne\c{s} Akta\c{s}}
\address{
 Department of Mathematics, Bilkent University\\
 06800 Ankara, Turkey} \email{cisem@fen.bilkent.edu.tr}

\subjclass[2000]{Primary 14J28; Secondary  14J10, 14J17}

\keywords{Complex quartic, singular quartic, $K3$-surface, simple singularity}

\date{}

\dedicatory{}

\begin{abstract}

We study complex spatial quartic surfaces with simple singularities up to equisingular deformations; as a first step, give a complete equisingular deformation classification of the so-called non-special simple quartic surfaces.

\end{abstract}

\maketitle

\section{Introduction}\label{introduction}
\subsection{Motivation}
Throughout the paper, all algebraic varieties are over the field $\mathbb{C}$ of complex numbers. A \emph{quartic } is a surface in $\mathbb{P}^3$ of degree $4$. We confine ourselves to \textit{simple} quartics only, \textit{i.e.}, those with $\mathbf{A}$--$\mathbf{D}$--$\mathbf{E}$ type singularities (see \S 3.1). Two such quartics are said to be \textit{equisingular deformation equivalent} if they belong to the same deformation family in which the total Milnor number stays constant.

The theory of simple spatial quartics is very similar to theory of simple plane sextics: both are closely related to $K3$-surfaces. It is well known that the minimal resolution of singularities of a simple  quartic $X\subset\mathbb{P}^3$ is a $K3$-surface. By using the global Torelli theorem for $K3$-surfaces \cite{K3}  and the surjectivity of the period map \cite{periodmap} the equisingular deformation classification of simple quartics can be reduced to the study of a certain arithmetic question about lattices. The corresponding counterpart for plane sextics is covered by  A.~Akyol and A.~Degtyarev~\cite{Alex2} (see also \cite{Alex1}), who completed the equisingular deformation classification of irreducible singular plane sextics.

Found in the literature are a great number of papers dealing with quartics in $\mathbb{P}^3$ and based on the theory of $K3$-surfaces and  V. V. Nikulin's results \cite{Niku2} on lattice extensions. For example, T. Urabe~\cite{Urabe2, Urabe1} listed  (in terms of perturbations of Dynkin graphs) all sets of singularities with the total Milnor number $\mu\leq 17$ that are realized by  simple quartics. He also showed in \cite{Urabe2} that $\mu\leq19$ for a simple quartic. Note that the classification of non-simple quartics with isolated singularities (which is not related to $K3$-surfaces) is complete: a complete list and a description in terms of lattice embeddings are found in A. Degtyarev~\cite{Alex3, Alex4}, and a description of some realizable sets of singularities in terms of Dynkin diagrams is found in T. Urabe~\cite{Urabe4,Urabe3}.

Also worth mentioning are various $K3$-related deformation classification problems dealing with real surfaces and other polarizations; see, e.g., \cite{AIK, Itemberg, Niku3} and the survey \cite{AK} for further references. Specifically, real quartics in $\mathbb{P}^3$ have been addressed by V.~Kharlamov~\cite{Kharlamov} (the classification of nonsingular real quartics up to rigid isotopy) and A.~Degtyarev, I.~Itenberg~\cite{AI} (arrangements of the ten nodes of a generic real determinantal quartic).


\subsection{Principal results}
In this paper we confine ourselves  to the so-called \emph{non-special} quartics (the precise definition is too technical to be stated here and we refer \S \ref{classquartics}). The counterpart of this notion in the realm of plane sextics are irreducible sextics admitting no dihedral coverings, \emph{cf}. \cite{Alex2}.  As yet another motivation, we have the following geometric characterization.
\begin{theorem}\label{motivation}
 A simple quartic $X\subset\mathbb{P}^3$ is non-special if and only if
 \begin{align*}
        H_1( X\smallsetminus(\operatorname{Sing}X\cup H))=0,
 \end{align*}
 where $\operatorname{Sing}X$ is the set of the singular points of $X$ and $H$ is a generic hyperplane section of $X$.
\end{theorem}
This theorem is proved in \S\ref{proof.motivation}.

A set of simple singularities can be identified with a root system, \textit{i.e.}, a negative definite lattice generated by vectors of square $-2$ (see Dufree~\cite{Dufree} and \S\ref{QuarticsandK3}). By a \textit{perturbation} of a set of simple singularities $\s$ we mean any set of simple singularities $\s'$ whose Dynkin graph is an induced subgraph of that of $\s$ (see \S\ref{proof.maintheorem}). Recall that for a simple quartic $X\subset\mathbb{P}^3$, one has $\mu(X)\leq19$ (see \eg, \cite{Urabe2}); $X$ is called \emph{maximizing} if $\mu(X)=19$.

Denote by $\mathcal{M}(\s)$ the equisingular stratum of simple quartics with a given set of singularities $\s$.  A connected component $\mathcal{D}\subset\mathcal{M}(\s)$ is called \textit{real} if it is preserved as a set under the complex conjugation map $\operatorname{conj}:\mathbb{P}^3\rightarrow\mathbb{P}^3$. Clearly, this property is independent of the choice of coordinates in $\mathbb P^3$, and all components of $\mathcal{M}(\s)$ split into real and pairs of complex conjugate ones.

Our principal result is a complete description of the equisingular strata $\mathcal{M}_1(\s)$ of non-special simple quartics.
\begin{table}\label{maxtable}
\caption{\small The space $\mathcal{M}_1(\s)$ with $\mu(\s)=19$}
\input{max_table}
\end{table}
\begin{table}\label{extremaltable}
\caption{\small Extremal sets of singularities with $\mu(\s)=18$ }
\input{table18}
\end{table}

\begin{theorem}\label{maintheorem}
A set of singularities $\s$ is realizable as the set of singularities of a non-special simple quartic if and only if $\s$ can be obtained by a perturbation from one of those listed in Tables $1$ and $2$. The numbers $(r,c)$ of, respectively, real and pairs of complex conjugate components of the strata $\mathcal{M}_1(\s)$ with $\mu(\s)=19$ are shown in Table $1$. If $\s$ is one of
\begin{align*}
  \mathbf{D}_6\oplus2\mathbf{A}_6,\quad\mathbf{D}_5\oplus2\mathbf{A}_6\oplus\mathbf{A}_1,\quad2\mathbf{A}_7\oplus2\mathbf{A}_2,
  \quad3\mathbf{A}_6,\quad2\mathbf{A}_6\oplus2\mathbf{A}_3\,
  \end{align*}
then $\mathcal{M}_1(\s)$ consists of two complex conjugate components; in all other cases, the stratum $\mathcal{M}_1(\s)$ is connected.
\end{theorem}
Theorem \ref{maintheorem} is proved in \S \ref{proof.maintheorem}.

\subsection{Contents of the paper}
Our principal result, Theorem \ref{maintheorem}, is proved by a reduction to an arithmetical problem \cite{AI} (\emph{cf.} also \cite{Alex1}), followed by Nikulin's theory of lattice extensions via discriminant groups \cite{Niku2}, Nikulin's existence theorem \cite{Niku2}, and Miranda--Morrison theory \cite{MM1, MM2,MM3} computing the genus groups and a few other bits missing in \cite{Niku2} in the case of indefinite lattices.

In $\S2$, based on Nikulin's work \cite{Niku2}, we recall the basic notions and results about integral lattices, discriminant forms and lattice extensions; then, we outline the fundamentals of Miranda-Morison's theory \cite{MM3} which are used in \S \ref{proof.maintheorem}. In $\S3$, we discuss the relation between simple quartics and $K3$-surfaces, explain the notion of abstract homological type, and recall the reduction of the classification problem to the arithmetical classification of abstract homological types. Finally, $\S4$ is devoted to the proofs of our principal results: the proof of Theorem\ref{maintheorem} is purely homotopy theoretical, whereas that of Theorem \ref{maintheorem} depends essentially on the auxiliary material presented in \S\ref{preliminaries} and \S\ref{quartics}.
\subsection{Acknowledgements}
I would like to express my gratitude to my advisor Alex Degtyarev for attracting my attention to the problem, motivating discussions, encouragement and infinite patience. I am also thankful to him for sharing his results (stated in Table $1$) about the moduli space of maximizing non-special simple quartics .

\section{Preliminaries}\label{preliminaries}

\subsection{Finite quadratic forms}\label{finite.quadratic.forms}

A \textit{finite quadratic form} is a finite abelian group $\mathcal{L}$ equipped with a map $q\colon\mathcal{L}\rightarrow \mathbb{Q}/2\mathbb{Z}$. satisfying $q(x+y)=q(x)+q(y)+2b(x,y)$ and $q(nx)=n^2x$ for all $x,y\in \mathcal{L}$, where $b\colon\mathcal{L}\otimes\mathcal{L}\rightarrow\mathbb{Q}/\mathbb{Z}$ is a symmetric bilinear form (which is determined by $q$). To reduce the notation we write $x^2$ for $q(x)$ and $x\cdot y$ for $b(x,y)$. For a prime $p$, let $\mathcal{L}_p:=\mathcal{L}\otimes\mathbb{Z}_p$, which is called the \emph{$p$-primary part} of $\mathcal{L}$. Any finite quadratic form $\mathcal{L}$ can be written as an orthogonal sum of its $p$-primary components $\mathcal{L}_p$, \emph{i.e.}, $\mathcal{L}=\bigoplus_p\mathcal{L}_p$ where the summation runs over all primes $p$. Denote by  $\ell(\mathcal{L})$ the minimal number of generators of $\mathcal{L}$.

Consider a fraction $\frac{m}{n}\in \mathbb{Q}/2\mathbb{Z}$ with $g.c.d(m,n)= 1$ and $mn=0$ mod $2$. By $\langle\frac{m}{n}\rangle$, we denote the finite non-degenerate (see \S\ref{integral.lattices}) quadratic form on $\mathbb{Z}/n\mathbb{Z}$ generated by an element of square $\frac{m}{n}$ and of order $n$. For an integer $k\geq 1$, let $\mathcal{U}(2^k)$ and $\mathcal{V}(2^k)$ be the quadratic forms on $\mathbb{Z}/2^k\mathbb{Z}\oplus\mathbb{Z}/2^k\mathbb{Z}$ defined by the matrices
\begin{align*}
    \mathcal{U}(2^k)=\left(
             \begin{array}{cc}
               0 & \frac{1}{2^k} \\
               \frac{1}{2^k} & 0 \\
             \end{array}
           \right)\qquad \quad\mbox{and}\qquad \quad\mathcal{ V}(2^k)=\left(
             \begin{array}{cc}
              \frac{1}{2^{k-1}} & \frac{1}{2^k} \\
               \frac{1}{2^k} & \frac{1}{2^{k-1}} \\
             \end{array}\right).
\end{align*}
Nikulin \cite{Niku2} showed that any finite quadratic form can be written as an orthogonal sum of cyclic summands of the form $\langle\frac{m}{n}\rangle$ and copies of $\mathcal{U}(2^k)$ and $\mathcal{V}(2^k)$.

The \textit{Brown invariant} of a finite quadratic form
$\mathcal{L}$ is the residue $\Br\mathcal{L}\in
\mathbb{Z}/8\mathbb{Z}$ defined by the Gauss sum
$$\exp\left(\frac{1}{4} i\pi \operatorname{Br}\mathcal{L}\right)=
|\mathcal{L}|^{-\frac{1}{2}}\sum\limits_{i\in \mathcal{L}} \exp(i\pi x^{2}).$$ The Brown invariants of indecomposable  $p$-primary blocks are as
follows:

$\Br\langle \frac{2a}{p^{2s-1}}
\rangle=2(\frac{a}{p})-(\frac{-1}{p})-1$,
~~~~$\Br\langle\frac{2a}{p^{2s}}\rangle=0$~~~~(for $p$ odd, $s\geq 1$ and
$\operatorname{g.c.d.}(a,p)=1$),

$\Br\langle\frac{a}{2^{k}} \rangle=a+\frac{1}{2} k(a^{2}-1) \mod
8$ ~~~~~~(for $k\geq 1$ and odd $a\in \mathbb{Z}$),

$\Br\mathcal{U}_{2^{k}}=0$,

$\Br\mathcal{V}_{2^{k}}=4k \mod
8$ ~~~~~~(for all $k\geq 1$).

A finite quadratic form is called \emph{even} if $x^2=0$ mod $\mathbb{Z}$ for all elements $x\in\mathcal{L}$ of order two; otherwise it is called \emph{odd}. This definition implies that a quadratic form is odd if and only if it contains $\langle\pm\frac{1}{2}\rangle$ as an orthogonal summand.

\subsection{Integral lattices and discriminant forms}\label{integral.lattices}

An \emph{(integral) lattice} is a free abelian group $L$ of finite rank with a symmetric bilinear form $b\colon L\otimes L\rightarrow \mathbb{Z}$. For short, we use the multiplicative notation $x\cdot y$ for $b(x,y)$ and $x^2$ for $b(x,x)$. A lattice $L$ is called \emph{even} if $a^2$ is an even integer for all $a\in L$. It is called \emph{odd} otherwise. The determinant $\operatorname{det}(L)$ is defined to be the determinant of the Gram matrix of $b$ in any basis of $L$. Since the transition matrix between any two integral bases has determinant $\pm 1$, $\operatorname{det}(L)\in \mathbb{Z}$ is well defined. A lattice $L$ is called \emph{non-degenerate} if $\operatorname{det}(L)\neq 0$; it is called \emph{unimodular} if $\det(L)=\pm 1$.

From now on all lattices considered are even and non-degenerate.

Given a lattice $L$, the form $b:L\otimes L \rightarrow \mathbb{Z}$ can be extended by linearity to a form $(L\otimes\mathbb{Q})\otimes_{\mathbb{Q}}(L\otimes\mathbb{Q})\rightarrow \mathbb{Q}$. If $L$ is non-degenerate, the dual group $L^*:= \operatorname{Hom}(L,\mathbb{Z})$ can be identified with the subgroup
\begin{align*}
    \bigl\{x\in L\otimes\mathbb{Q}\ \bigm|\ \text{$x\cdot y\in \mathbb{Z}$ for all $x\in L$}  \bigr\}
\end{align*}

Since the original bilinear form $b$ on $L$ is integer valued, $L$ is a finite index subgroup of its dual. The quotient $L^*/L$ is called the \emph{discriminant group} of $L$ and is denoted by $\mathcal{L}$ or $\operatorname{disc} L$. If $\{e_1,e_2,\ldots e_n\}$ is a basis set for $L$ and $ \{e_1^*,e_2^*, \ldots,e_n^*\}$ is the dual basis for $L^*$, then the Gram matrix $[e_i\cdot e_j]$ is exactly the matrix of the homomorphism $\varphi: L\rightarrow L^*$, $x\mapsto[y\mapsto x\cdot y]$. Hence one has $|\mathcal{L}|=|\det (L)|$.  Note that $x\cdot y \in\mathbb{Z}$ whenever $x\in L$ or $y\in L$. Thus, $\mathcal{L}$ inherits from $L\otimes \mathbb{Q}$ a non-degenerate symmetric bilinear form $b_{\mathcal{L}}: \mathcal{L}\otimes\mathcal{L}\rightarrow\mathbb{Q}/\mathbb{Z}$; it is  called the \emph{discriminant form}. If $L$ is even, this form $b_{\mathcal{L}}$ can be promoted to the quadratic extension $q_{\mathcal{L}}: \mathcal{L}\rightarrow\mathbb{Q}/2\mathbb{Z}$, $x\bmod L\mapsto x^2 \bmod 2\mathbb{Z}$. Hence, the discriminant form of an even lattice is a finite quadratic form. Accordingly, given a prime $p$, we use the notation $\operatorname{disc}_p L$ or $\mathcal{L}_p$ for the $p$-primary part of $\mathcal{L}$, \emph{i.e.}, $\mathcal{L}_p =\mathcal{L}\otimes \mathbb{Z}_p$. Each discriminant group $\mathcal{L}$ decomposes into orthogonal sum $\mathcal{L}=\bigoplus_p\mathcal{L}_p$ of its $p$-primary components.\\
\indent The \emph{signature} of a non-degenerate lattice $L$ is the pair $(\sigma_+,\sigma_-)$ of its positive and negative inertia indices. Two non-degenerate integral lattices are said to have the same \emph{genus} if their localizations over $\mathbb{R}$ and over $\mathbb{Q}_p$ are isomorphic. The following few statements give the relation between the genus of an even integral lattice and its discriminant form.

\begin{theorem}[Nikulin \cite{Niku2}]
The genus of an even integral lattice $L$ is determined by its signature $(\sigma_+ L, \sigma_- L)$ and discriminant form $\operatorname{disc} L$.
\end{theorem}
The existence of an even integral lattice $L$ with a given signature is given by Nikulin's existence theorem (see Theorem \ref{Existence}).
\begin{theorem}[van der Blij \cite{Blij}]
For any non-degenerate even integral lattice $L$ one has $\operatorname{Br}\mathcal{L}=\sigma_+-\sigma_ -\mod 8$.
\end{theorem}
We denote by $g(L)$ the set of all isomorphism classes of all non-degenerate even integral lattices with the same genus as $L$. Each set $g(L)$ is known to contain finitely many isomorphism classes.

Given a prime $p$, we define the \textit{determinant} $\operatorname{det_p}(\mathcal{ L})$ as the determinant of the matrix of the quadratic from on $\mathcal{L}_p$ in an appropriate basis (see \cite{Niku1} and \cite{Niku2} for details). Unless $p=2$ one has $\operatorname{det}_p(\mathcal{L})=u/|\mathcal{L}_p|$ where $u$ is a well defined element of $u\in\mathbb{Z}_p^\times/(\mathbb{Z}_p^\times)^2$. If $p=2$,  the determinant $\operatorname{det}_2(\mathcal{L})$ is well defined only if $\mathcal{L}_2$ is even.

\begin{theorem}[Nikulin \cite{Niku2}]{\label{Existence}}
Let $\mathcal{L}$ be a finite quadratic form
and let $\sigma_{\pm}$ be a pair of integers. Then, the following
four conditions are necessary and sufficient for the existence of
an even integral lattice $L$ whose signature is
$(\sigma_{+},\sigma_{-})$ and whose discriminant form is
$\mathcal{L}$:

\begin{enumerate}

\item $\sigma_{\pm} \geq 0$ and $\sigma_{+}+\sigma_{-} \geq
\operatorname{\ell} (\mathcal{L})$;

\item $\sigma_{+}-\sigma_{-}=\operatorname{Br}\mathcal{L}
\mod 8$;

\item for each $p\neq 2$, either
$\sigma_{+}+\sigma_{-}>\ell_{p}(\mathcal{L})$ or
$\operatorname{det}_p(\mathcal{L})\equiv (-1)^{\sigma_{-}}\cdot
|\mathcal{L}|\bmod (\mathbb{Z}_{p}^{*})^{2}$;

\item either $\sigma_{+}+\sigma_{-}>\ell_{2}(\mathcal{L})$, or
$\mathcal{L}_{2}$ is odd, or
$\operatorname{det}_2(\mathcal{L})\equiv \pm
|\mathcal{L}|\bmod (\mathbb{Z}_{2}^{*})^{2}$.
\end{enumerate}

\end{theorem}










\subsection{\textbf{Automorphisms of lattices}}\label{automorphisms.of.lattices}

An \emph{isometry} of integral lattices is a homomorphism of abelian groups preserving the forms. The group of auto-isometries of $L$ is denoted by $O(L)$. There is a natural homomorphism $d\colon O(L)\rightarrow \Aut(\mathcal{L})$, where $\Aut(\mathcal{L})$ denotes the group of automorphisms of $\mathcal{L}$ preserving the discriminant form $q$ on $\mathcal{L}$. Obviously, one has $\operatorname{Aut}(\mathcal{L})=\prod_p\operatorname{Aut}(\mathcal{L}_p)$, where the product runs over all primes. The restrictions of $d$ to the $p$-primary components are denoted by $d_p: O(L)\rightarrow\operatorname{Aut}(\mathcal{L}_p)$.

Given a vector $u$ in $L$ with $u\neq 0$, the \emph{reflection} against its orthogonal hyperplane is the automorphism
\begin{align*}
    r_u:& L\rightarrow L\\
       & x\mapsto x-2\frac{(x\cdot u)}{u^2}u
\end{align*}
The reflection $r_u$ is well-defined whenever $u\in(\frac{u^2}{2})L^*$. Note that $r_u^2=\operatorname{id}$, \emph{i.e.}, $r_u$ is an involution. Each image $d_p(r_u)\in \Aut (\mathcal{L}_p)$ is also a reflection (see \S \ref{reflections}). If $u^2=\pm 1$ or $u^2=\pm 2$, then the induced automorphism $d(r_u)$ is the identity.

\subsection{Lattice extensions}\label{lattice.extensions}

An even integral lattice $L$ containing even lattice $S$ called an \emph{extension} of $S$. As \emph{isomorphism} between two extensions $L_1\supset S$ and $L_2\supset S$ is an isometry between $L_1$ and $L_2$ taking $S$ to $S$. In particular, if the isomorphism $L_1\rightarrow L_2$ restricts to $\operatorname{id}$ on $S$, the extensions $L_1$ and $L_2$ are called \emph{strictly isomorphic}. For a given subgroup $A$ of O(S), we define \emph{$A$-isomorphisms} of extensions of $S$ as those which restrict to an element of $A$ on $S$.

Recall that $S$ is assumed non-degenerate, hence given a \emph{finite index}  extension $L\supset S$, one has $L\subset S^{\ast}$. Thus there are inclusions $S\subset L\subset L^*\subset S^*$ which imply $L/S\subset S^*/S=\mathcal{S}$. The subgroup $\mathcal{K}=L/S$ of $\mathcal{S}$ is called the \emph{kernel} of the finite index extension $L\supset S$. Since $L$ is an even integral lattice, the discriminant quadratic form on $\mathcal{S}$ restricts to zero on $\mathcal{K}$, \emph{i.e.,} $\mathcal{K}$ is isotropic.

\begin{proposition}[Nikulin \cite{Niku2}]
Let $S$ be a non-degenerate even lattice, and fix a subgroup $A\subset
O(S)$. The map $L\mapsto \mathcal{K}=L/S \subset \mathcal{S}$
establishes a one-to-one correspondence between the set of
$A$-isomorphism classes of finite index extensions $L\supset S$
and the set of $A$-orbits of isotropic subgroups
$\mathcal{K}\subset \mathcal{S}$. Under this correspondence one
has $L=\{x\in S^*\,|\, (x \bmod S)\in\mathcal{K}\}$ and $\mathcal{L}=\mathcal{K}^{\bot}/\mathcal{K}$.
\end{proposition}

\begin{proposition}[Nikulin \cite{Niku2}]
Let $L\supset S$ be a finite index extension of a lattice $S$ and let $\mathcal{K}\subset \mathcal{S}$ be its kernel. Then an auto-isometry $S\rightarrow S$ extends to $L$ if and only if the induced automorphism of $\mathcal{S}$ preserves $\mathcal{K}$.
\end{proposition}

An extension $L\supset S$ is called \textit{primitive} if $L/S$ is torsion free. Following Nikulin \cite{Niku2}, we confine ourselves to the special case where $L$ is unimodular. If $S$ is a primitive non-degenerate sublattice of a unimodular lattice $L$ then $S^{\bot}$ is also primitive in $L$ and $L$ is a finite index extension of $S\oplus S^{\bot}$. Furthermore, since $\operatorname{disc}L=0$, the kernel $\mathcal{K}\subset\mathcal{S}\oplus \mathcal{S}^{\bot}$ is the graph of an anti-isometry $\psi\colon \mathcal{S}\rightarrow \disc S^\bot$. Hence the genus $g(S^{\bot})$ is determined by the genera $g(N)$ and $g(L)$. Conversely, given a lattice $N\in g(S^{\bot})$ and an anti-isometry $\psi\colon\mathcal{S}\rightarrow\mathcal{N}$, the graph of $\psi$ is an isotropic subgroup $\mathcal{K}\subset\mathcal{S}\oplus \mathcal{S}^{\bot}$ and the corresponding finite index extension $S\oplus N\hookrightarrow L$ is a unimodular primitive extension of $S$ with $S^{\bot}\cong N$. Note that an anti-isometry $\psi\colon \mathcal{S}\rightarrow \disc S^\bot$ induces a homomorphism $d^{\psi}\colon O(S)\rightarrow\Aut(\mathcal{N})$. Thus, since also an indefinite unimodular lattice is unique in its genus, we have the following theorem.
\begin{theorem}[Nikulin \cite{Niku2}]\label{Nikulin.doublecoset}
Let $L$ be an indefinite unimodular even lattice and $S\subset L$ a non-degenerate primitive sublattice. Fix a subgroup $A\subset O(S)$. Then the $A$-isomorphism class of a primitive extension $S\subset L$ is determined by
\begin{enumerate}
  \item a choice of a lattice $N\in g(S^{\bot})$ and
  \item a choice of a double coset $c_N \in d^{\psi}(A)\backslash Aut (\mathcal{N})/ \operatorname{Im} d$  (for a given $N$ and some anti-isometry $\psi:\mathcal{S}\rightarrow\mathcal{N}$ inducing $d^{\psi}$).
\end{enumerate}
\end{theorem}
\begin{theorem}[Nikulin \cite{Niku2}]\label{Nikulin.homological}
Let $L$ be an indefinite unimodular even lattice, $S\subset L$ a non-degenerate primitive sublattice and $\psi:\mathcal{S}\rightarrow\mathcal{N}$ the anti-isometry where, $N=S^{\bot}$. Then a pair of isometries $a_S\in O(S)$ and $a_N\in O(N)$ extends to $L$ if and only if $d^{\psi}(a_S)=d(a_N)$.
\end{theorem}


\subsection{Miranda--Morrison's theory}\label{sectionMM}
Let $p$ be a prime. Define
 \begin{align*}
 \Gamma_p:&=\{\pm1\}\times\mathbb{Q}^{\times}_p/(\mathbb{Q}^{\times}_p)^2,\\
 \Gamma_0:&=\{\pm1\}\times\{\pm1\}\subset\{\pm1\}\times\mathbb{Q}^{\times}/(\mathbb{Q}^{\times})^2.
 \end{align*}
It is convenient to introduce the following subgroups related to $\Gamma_p$ :
\begin{itemize}
\item $\Gamma_{p,0}:=\{(1,1),(1,u_p),(-1,1),(-1,u_p)\}\subset\Gamma_p$; here, $p$ is odd and $u_p$ is the only nonzero element of $\mathbb{Z}^{\times}_p/(\mathbb{Z}^{\times}_p)^2$,
\item $\Gamma_{2,0}:=\{(1,1),(1,3),(1,5),(1,7),(-1,1),(-1,3),(-1,5),(-1,7)\}\subset\Gamma_2$,
\item $\Gamma_p^{++}:=\{1\}\times\mathbb{Z}_p^{\times}/(\mathbb{Z}_p^{\times})^2\subset\Gamma_{p,0}$,
\item $\Gamma_{2,2}:=\{(1,1),(1,5)\}\subset\Gamma_2^{++}$,
\item $\Gamma'_{2,0}:=\Gamma_{2,0}/\Gamma_{2,2}$ (and $\Gamma'_{p,0}:=\Gamma_{p,0}$ for $p\neq2$),
\item $\Gamma_0^{--}:=\{(1,1),(-1,-1)\}\subset\Gamma_0$.
\end{itemize}
Let, further,
\begin{align*}
    \Gamma_{\mathbb{A},0}:=\prod_p\Gamma_{p,0}\subset\Gamma_{\mathbb{A}}:=\Gamma_{\mathbb{A},0}\cdot\sum_p\Gamma_p
\end{align*}
where $"\cdot"$ denotes the sum of the subgroups. Note that
\begin{align*}
    \Gamma_{\mathbb{A}}=\{(d_p,s_p)\in \prod_p\Gamma_p\ |(d_p,s_p)\in\Gamma_{p,0}\text{ for almost all $p$}\}
\end{align*}
The natural map $\mathbb{Q}^{\times}/(\mathbb{Q}^{\times})^2\rightarrow\mathbb{Q}^{\times}_p/(\mathbb{Q}^{\times}_p)^2$ induces canonical maps
\begin{align}\label{fayp}
     \varphi_p:\Gamma_0\rightarrow\Gamma_{p,0}.
\end{align}

Let $N$ be an indefinite lattice with $\operatorname{rk}(N)\geq 3$. We will use certain subgroups $\Sigma_p^{\sharp}(N)\subset\Gamma_{p,0}$ and $\Sigma_p(N)\subset\Gamma_p$. In the notation of \cite{MM3} (which slightly differs from the notation in \cite{MM1, MM2}), one has $\Sigma_p^{\sharp}(N):= \Sigma^{\sharp}(N\otimes\mathbb{Z}_p)$ and $\Sigma_p(N):=\Sigma(N\otimes\mathbb{Z}_p)$; we refer the reader to \cite{MM3} for the precise definitions. The subgroups $\Sigma^{\sharp}_p(N)$ are computed explicitly in \cite{MM3} (see Theorem 12.1, 12.2, 12.3 and 12.4).

Also defined in \cite{MM3} is the $\mathbb{F}_2$-module
\begin{align}\label{E(N)}
    E(N):=\Gamma_{\mathbb{A},0}/\prod_p\Sigma_p^{\sharp}(N)\cdot\Gamma_0.
\end{align}
This module is finite. Indeed, following \cite{MM3}, we call a prime $p$ \emph{regular} with respect to $N$ if $\Sigma^{\sharp}_p(N)=\Gamma_{p,0}$. Crucial is the fact that  a prime $p$ is regular unless $p\mathrel|\operatorname{det}(N)$; thus, (\ref{E(N)}) reduces to finitely many primes $p$:
\begin{align}\label{E(N)finite}
    E(N)=\Gamma_{\mathbb{A},0}/\prod_{p|\operatorname{det}(N)}\Sigma_p^{\sharp}(N)\cdot\Gamma_0.
\end{align}

\begin{theorem}[Miranda--Morrison \cite{MM3}]\label{MMexact.sequence}
Let $N$ be a non-degenerate indefinite even lattice with $\operatorname{rk}(N)\geq 3$. Then there is an exact sequence
\begin{align}\label{exactsequence}
    O(N)\xrightarrow{d}\operatorname{Aut}(\mathcal{N})\xrightarrow{\mathrm{e}} E(N)\rightarrow g(N)\rightarrow 1,
\end{align}
where $g(N)$ is the genus group of $N$.
\end{theorem}
A simplified version of (\ref{E(N)finite}) computing  the numeric invariants
\begin{align*}
    e_p(N):=[\Gamma_{p,0}:\Sigma^{\sharp}(N)] \text{ and } \tilde{\Sigma}_p(N):=\varphi^{-1}_p(\Sigma^{\sharp}_p(N))\subset\Gamma_0,
\end{align*}
is found in \cite{MM1, MM2}. This gives us the size of the group $E(N)$: one has
\begin{align}\label{orderofEN}
    |E(N)|=\frac{e(N)}{[\Gamma_0:\tilde\Sigma(N)]}
\end{align}
where
\begin{align*}
    e(N):=\prod_p e_p(N), \quad\tilde\Sigma(N):=\bigcap_p \tilde\Sigma_p(N),
\end{align*}
and the product and intersection run over all primes $p$ or, equivalently, over all primes $p\mathrel|\operatorname{det}(N)$.

The following theorem can be deduced from Theorems \ref{Nikulin.doublecoset} and \ref{MMexact.sequence}.
\begin{theorem}[Miranda--Morrison \cite{MM1,MM2}]\label{correspondence.E(N)}
Let $S$ be a primitive sublattice of an even unimodular lattice $L$ such that $N:=S^{\bot}$ is a non-degenerate indefinite even lattice with $\operatorname{rk}(N)\geq3$. Then the strict isomorphism classes of primitive extensions $S\hookrightarrow L$ are in a canonical one-to-one correspondence with the group $E(N)$.
\end{theorem}
As explained \S\ref{lattice.extensions}, given a unimodular lattice $L$ and a primitive sublattice $S\subset L$, one has an anti-isometry $\psi :\mathcal{S}\rightarrow \mathcal{N}$ (where $N=S^{\perp}$), which induces a homomorphism $d^{\psi}:O(S)\rightarrow\operatorname{Aut}(\mathcal{N})$. If $N$ is indefinite and $\operatorname{rk}(N)\geq 3$, then $d(O(S))\subset\Aut(\mathcal{N})$ is a normal subgroup with abelian quotient (see (\ref{exactsequence})) and  we have a homomorphism  $d^{\bot}:O(S)\rightarrow \Aut(\mathcal{N})\xrightarrow{\mathrm{e}}E(N)$ independent of the choice of an anti-isometry $\psi$. The next statement follows from Theorems \ref{MMexact.sequence} and \ref{Nikulin.doublecoset}.
\begin{corollary}\label{coker}
Let $S$ be a primitive sublattice of an even unimodular lattice $L$ such that $N:=S^{\bot}$ is non-degenerate indefinite even lattice with $rk(N)\geq3$ and let $A\subset O(S)$ be a subgroup. Then, the $A$-isomorphism classes of primitive extensions $S\hookrightarrow L$ are in a one-to-one correspondence with the $\mathbb{F}_2$-module $\operatorname{coker}d^\bot(A)$.
\end{corollary}

\subsection{Reflections}\label{reflections}
Recall that $\Aut(\mathcal{N})=\prod_p\Aut (\mathcal{N}_p)$ where $p$ runs over all primes. Let $s$ be a prime and $\alpha\in\mathcal{N}_s$ such that
\begin{equation}\label{star}
  \text{$s^k\alpha=0$ and $\alpha^2=\frac{2u}{s^k} \bmod2\mathbb{Z}$, g.c.d$(u,s)=1$, $k\in\mathbb{N}$.}
\end{equation}
We denote by $\mathcal{N}^{\dag}_s$ the set of all elements $\alpha\in\mathcal{N}_s$ satisfying (\ref{star}) and let $\mathcal{N}^{\dag}=\bigcup_s\mathcal{N}^{\dag}_s$.
Then one can define a map,
\begin{align*}
  \mathcal{N}_s\rightarrow\mathbb{Z}/s^k,\quad   x\mapsto\frac{2(x\cdot \alpha)}{\alpha^2}\bmod s^k.
\end{align*}
where $\alpha\in\mathcal{N}^{\dag}_s$. Thus, there is a reflection $r_{\alpha}\in\Aut \mathcal{N}_s$ given by
\begin{align*}
    r_{\alpha}: x\mapsto x-\frac{2(x\cdot \alpha)}{\alpha^2}\alpha.
\end{align*}
If $\alpha^2=\frac{1}{2}\mod\mathbb{Z}$ and $2\alpha=0$ then $r_{\alpha}=\operatorname{id}$.

Let $p$ be a prime and consider the homomorphism
\begin{align*}
    \Aut(\mathcal{N})=\prod_p\Aut(\mathcal{N}_p)\xrightarrow{\phi}\prod_p\Sigma_p(N)/\Sigma^{\sharp}_p(N)
\end{align*}
which is the product of the epimorphisms
\begin{align*}
    \phi_p:\Aut(\cN_p)\twoheadrightarrow\Sigma_p(N)/\Sigma^{\sharp}_p(N)
\end{align*}
introduced in Miranda--Morrison \cite{MM3}. The images of the homomorphism $\phi_p$ can be computed on reflections as follows: For a prime $s$ and an element $\alpha\in\cN^{\dag}_s$,
the image of the reflection  $r_{\alpha}\in\Aut(\cN_s)$ under $\phi_s$ is given by $\phi_s(r_{\alpha})=(-1,us^k)$, see (\ref{star}).
If $s=2$ and $\alpha^2=0\bmod \mathbb{Z}$, then $\phi_s(r_{\alpha})$ is only well-defined $\mod \Gamma_2^{++}$. If $s=2$ and $\alpha^2=\frac{1}{2}\bmod \mathbb{Z}$, then $\phi_s(r_{\alpha})$ is well-defined $\mod {\Gamma_{2,2}}$. In these cases to determine the value of $\phi_s(r_{\alpha})$, we need more information about $\alpha$ and $N$.

 Given another prime $p$, we define the \emph{$p$-norm} $|\alpha|_p\in \{\pm1\}$ of $\alpha\in \cN^{\dag}_s$ by
$$
|\alpha|_p := \left\{
        \begin{array}{ll}
            \chi_p(s^k) & \quad \text{if $s\neq p$}, \\
            \chi_p(u) & \quad \text{if $s=p$},
        \end{array}
    \right.
$$
where the homomorphism $\chi_p:\mathbb{Z}_p^\times/(\mathbb{Z}_p^\times)^2\rightarrow\{\pm1\}$ is defined as
$$
\chi_p(u) := \left\{
        \begin{array}{ll}
           \big(\frac{u}{p}\big) & \quad \text{if $p\neq 2$}, \\
           u \bmod 4 & \quad \text{if $p=2$}.
        \end{array}
    \right.
$$
Note that $|\alpha|_2$ is undefined when $p=2$ and $\alpha^2=0\mod \mathbb{Z}$. Following \cite{Alex2}, given  primes $p$ and $s$ and a vector $\alpha\in \cN^{\dag}_s$, we introduce the group

$$
E_p(N) := \left\{
       \begin{array}{ll}
      \{ \pm1\} & \quad \text{if $p=1\bmod 4$ and $e_p(N)\cdot|\tilde{\Sigma}_p(N)|=8$}, \\
        \quad1 & \quad \text{otherwise},
     \end{array}
   \right.
$$
the map $\bar{\phi}_p:\cN^{\dag}_s\rightarrow E_p(N)$,
$$
\bar{\phi}_p(\alpha):= \left\{
        \begin{array}{ll}
          1 & \quad \text{if $E_p(N)=1$}, \\
          |\alpha|_p & \quad \text{otherwise},
        \end{array}
    \right.
$$
and the map $\bar{\beta}_p:\cN^{\dag}_s\rightarrow \Gamma_0$,
$$
\bar{\beta}_p(\alpha):= \left\{
        \begin{array}{ll}
          (\delta_p(\alpha)\cdot|\alpha|_p,1) & \quad \text{if $p=1\bmod 4$}, \\
          \delta_p(\alpha)\times|\alpha|_p& \quad \text{otherwise},
        \end{array}
    \right.
$$
where the map
$$\delta_p(\alpha):=(-1)^{\delta_{p,s}}$$
(here  $\delta_{p,s}$ is the conventional Kronecker symbol). Note  that we have the assignment
\begin{align*}
    r_{\alpha}\mapsto(\delta_p(\alpha),|\alpha|_p)\in\Gamma'_{p,0}.
\end{align*}
The following lemmas provide an explicit description for the group $E(N)$ and compute the image of the  homomorphism $\mathrm{e}$ on the reflections $r_{\alpha}$ for the special case when $N$ has one or two irregular primes.
\begin{lemma}[Akyol--Degtyarev \cite{Alex2}]\label{1irr}
Let $N$ be a non-degenerate indefinite even lattice with $rk(N)\geq3$, $\Sigma_2^{\sharp}(N)\supset\Gamma_{2,2}$, and assume that $N$ has one irregular prime $p$. Then $E(N)=E_p(N)$ and $\mathrm{e}(r_\alpha)=\bar{\phi}_p(\alpha)$ for any $\alpha\in \cN^{\dag}$.
\end{lemma}
\begin{lemma}[Akyol--Degtyarev \cite{Alex2}]\label{2irr}
Let $N$ be a non-degenerate indefinite even lattice with $rk(N)\geq3$, $\Sigma_2^{\sharp}(N)\supset\Gamma_{2,2}$, and assume that $N$ has two irregular primes $p$, $q$. Then
\begin{align*}
E(N)&=E_p(N)\times E_q(N)\times(\Gamma_0/\tilde{\Sigma}_p(N)\cdot\tilde{\Sigma}_q(N)),\\
\mathrm{e}(r_{\alpha})&=\bar{\phi}_p(\alpha)\times\bar{\phi}_q(\alpha)\times(\bar{\beta}_p(\alpha)\cdot\bar{\beta}_q(\alpha)),
\end{align*}
for any $\alpha\in \cN^{\dag}$ such that $\alpha^2\neq0\bmod\mathbb{Z}$ if $p=2$ or $q=2$.
\end{lemma}

\begin{corollary}[Akyol--Degtyarev \cite{Alex2}]\label{c2irr}
Under the hypothesis of Lemma \ref{2irr}, assume, in addition, that $|E(N)|=|E_p(N)|=2$. Then $E(N)=E_p(N)$ and $\mathrm{e}(r_{\alpha})=|\alpha|_p$ for any $\alpha\in \cN^{\dag}$.
\end{corollary}
\subsection{Positive sign structure}\label{positive.sign.structure}
Let $N$ be a non-degenerate lattice. The orthogonal projection of any positive definite $2$-subspace in $N\otimes\mathbb{R}$ to any other such subspace is an isomorphism of vector spaces. Thus a choice of an orientation of one maximal positive definite subspace in $N\otimes\mathbb{R}$ defines a coherent orientation of any other. A choice of an orientation of a maximal positive definite subspace of $N\otimes\mathbb{R}$ is called a \textit{positive sign structure}. We denote by $O^+(N)$ the subgroup of $O(N)$ consisting of the isometries preserving a positive sign structure. Obviously either $O^+(N)=O(N)$ or $O^+(N) \subset O(N)$ is a subgroup of index $2$. In the latter case, each element of $O(N)\smallsetminus O^+(N)$ is called a $+$-\textit{disorienting} isometry of $N$. Following \cite{MM3}, we define the map $\operatorname{det}_+:O(N)\rightarrow \{\pm1\}$ as
$$
\operatorname{det}_+(a):= \left\{
        \begin{array}{ll}
          -1 & \quad \text{if $a$ preserves the positive sign structure}, \\
          +1& \quad \text{if $a$ reserves the positive sign structure}.
        \end{array}
    \right.
$$
Note that $\operatorname{Ker}(\operatorname{det}_+)=O^+(N)$.
\begin{proposition}[Miranda--Morrison \cite{MM3}]\label{propdet+}
Let $N$ be a non-degenerate indefinite even lattice with $\operatorname{rk}(N)\geq3$. Then one has $\tilde{\Sigma}(N)\subset\Gamma_0^{--}$ if and only if $\operatorname{det}_+(a)=1$ for all $a\in\operatorname{Ker}[d\colon O(N)\rightarrow \Aut(\mathcal{N})]$.
\end{proposition}
The following lemma computes the images of the function $\operatorname{det}_+$ on reflections.
\begin{lemma}[Akyol--Degtyarev \cite{Alex2}]\label{lemmadet+}
Let $N$ be a non-degenerate indefinite even lattice with $\operatorname{rk}(N)\geq3$, $\Sigma_2^{\sharp}(N)\supset\Gamma_{2,2}$, and assume that there is a prime $p$ such that $\tilde{\Sigma}_p(N)\subset\Gamma_{0}^{--}$. Then, for an element $\alpha\in \cN^{\dag}$ such that $r_{\alpha}\in\operatorname{Im}d$ and $\alpha^2\neq0\bmod \mathbb{Z}$ if $p=2$, one has $\operatorname{det_+}(r_{\alpha})=\delta_p(\alpha)\cdot|\alpha|_p$.
\end{lemma}
Defined in \cite{MM2}, we introduce the group
\begin{align}\label{E+(N)}
    E^+(N):=\Gamma_{\mathbb{A},0}/\prod_p\Sigma_p^{\sharp}(N)\cdot\Gamma_0^{--}.
\end{align}
(Similar to (\ref{E(N)}) and (\ref{E(N)finite}) the actual computation reduces to finitely many primes $p\mathrel|\operatorname{det}(N)$.)
As in Theorem \ref{MMexact.sequence} there is an exact sequence
\begin{align*}
    O^+(N)\xrightarrow{d}\operatorname{Aut}(\mathcal{N})\xrightarrow{\mathrm{e}^+}E^+(N)\rightarrow g(N)\rightarrow 1.
\end{align*}
The size of the group $E^+(N)$ is also computed in \cite{MM2}: one replaces $[\Gamma_0:\tilde\Sigma(N)]$ in (\ref{orderofEN}) with $[\Gamma_0^{--}:\tilde{\Sigma}(N)\cap\Gamma_0^{--}]$. For an irregular prime $p$, we denote $\tilde{\Sigma}^+_p(N):=\tilde{\Sigma}_p(N)\cap\Gamma_0^{--}$.

Given a unimodular even lattice $L$ and a primitive sublattice $S\subset L$ such that $N:=S^{\bot}$ is a non-degenerate indefinite lattice with $\operatorname{rk}(N)\geq 3$, we have a well-defined homomorphism $d^{\bot}_+\colon O(S)\rightarrow E^+(N)$, \emph{cf}. the definition of $d^{\bot}$ in \S\ref{sectionMM}.

Let $p$ and $s$ be two irregular primes and choose an element $\alpha\in \cN^{\dag}_s$ as in (\ref{star}), we introduce the group
$$
E^{+}_p(N) := \left\{
       \begin{array}{ll}
      E_p(N) & \quad \text{if $p=1\bmod 4$}, \\
       \Gamma_{0}/\tilde{\Sigma}_p(N)\cdot\Gamma_0^{--} & \quad \text{otherwise},
     \end{array}
   \right.
$$
the map $\bar{\phi}^+_p:\cN^{\dag}_s\rightarrow E^+_p(N)$,
$$
\bar{\phi}_p^+(\alpha):= \left\{
        \begin{array}{ll}
          \bar{\phi}_p(\alpha) & \quad  \text{if $p=1\bmod 4$}, \\
         \delta_p(\alpha)\cdot|\alpha|_p& \quad \text{if $p\neq1\bmod 4$ and $E^+_p(N)\neq1$},\\
           1& \quad \text{if $p\neq1\bmod 4$ and $E^+_p(N)=1$},
        \end{array}
    \right.
$$
and the map $\bar{\beta}_p^+:\cN^{\dag}_s\rightarrow \Gamma_0^{--}$,
$$
\bar{\beta}_p^+(\alpha):= \left\{
        \begin{array}{ll}
          \delta_p(\alpha)\cdot|\alpha|_p & \quad \text{if $p=1\bmod 4$}, \\
          |\alpha|_p& \quad \text{if $p\neq1\bmod 4$ and $E^+_p(N)\neq1$},\\
     \operatorname{proj}(\bar{\beta}_p(\alpha))   & \quad \text{if $p\neq1\bmod 4$ and $E^+_p(N)=1$},
        \end{array}
    \right.
$$
where $\operatorname{proj}:\Gamma_0\rightarrow\Gamma_0/\tilde{\Sigma}_p(N)=\Gamma_0^{--}$ is the projection map.  Next lemma computes the group $E^+(N)$ and the values of the homomorphism $\mathrm{e}^+$ on the reflections $r_{\alpha}$
\begin{lemma}[Akyol--Degtyarev \cite{Alex2}]\label{2irr+}
Let $N$ be a non-degenerate indefinite even lattice with $\operatorname{rk}(N)\geq3$, $\Sigma_2^{\sharp}(N)\supset\Gamma_{2,2}$ and assume that $N$ has two irregular primes $p$, $q$. Then
\begin{align*}
E^+(N)&=E_p^+(N)\times E^+_q(N)\times(\Gamma_0^{--}/\tilde{\Sigma}_p^+(N)\cdot\tilde{\Sigma}^+_q(N))\\
\mathrm{e}^+(r_{\alpha})&=\bar{\phi}^+_p(\alpha)\times\bar{\phi}^+_q(\alpha)\times(\bar{\beta}_p^+(\alpha)\cdot\bar{\beta}_q^+(\alpha))
\end{align*}
for any $\alpha\in\cN^{\dag}$ such that $\alpha^2\neq0\bmod\mathbb{Z}$ if $p=2$ or $q=2$.
\end{lemma}

\subsection{Root Systems}\label{root.systems}

A \textit{root} in a lattice $L$ is an element $v\in L$ of square $-2$. A \emph{root system} is a negative definite lattice generated by its roots.  Each root system splits uniquely into orthogonal sum of its irreducible components. As explained in \cite{Bour}, the irreducible root systems are $\mathbf{A}_n, n\geq 1$, $\mathbf{D}_m, m\geq4$ and $\mathbf{E}_6$, $\mathbf{E}_7$, $\mathbf{E}_8$. The corresponding discriminant forms are as follows:
\begin{gather*}
\disc \mathbf{A}_n=\Big\langle-\frac{n}{n+1}\Big\rangle,\quad \disc \mathbf{D}_{2k+1}=\Big\langle-\frac{2k+1}{4}\Big\rangle,\\
\disc\mathbf{D}_{8k\pm2}=2\Big\langle\mp\frac{1}{2}\Big\rangle,\quad \disc \mathbf{D}_{8k}=\mathcal{U}(2),\quad \disc \mathbf{D}_{8k+4}=\mathcal{V}(2),\\
\disc\mathbf{E}_6=\Big\langle\frac{2}{3}\Big\rangle,\quad \disc \mathbf{E}_7=\Big\langle\frac{1}{2}\Big\rangle,\quad \disc \mathbf{E}_7=0.
\end{gather*}
Given a root system $S$, the group generated by reflections (defined by the roots of $S$) acts simply transitively on the set of Weyl chambers of $S$. The roots constituting a single Weyl chamber form a \emph{standard basis} for $S$; these roots are naturally identified with the vertices of the Dynkin graph $\Gamma:=\Gamma_S$ Thus, ones has an obvious homomorphism
\begin{align*}
    \operatorname{Sym}(\Gamma)\rightarrow O(S)\rightarrow\Aut(\mathcal{S})
\end{align*}
where $\operatorname{Sym}(\Gamma)$ denotes the symmetries of $\Gamma$. Irreducible root systems correspond to connected Dynkin graphs. The following statement follows immediately from the classification of connected Dynkin graphs (see N. Bourbaki~\cite{Bour}).
\begin{lemma}
Let $\Gamma=\Gamma_S$ be the connected Dynkin graph of an irreducible root system $S$. Then,
\begin{enumerate}
     \item if $S$ is $\mathbf{A}_1$, $\mathbf{E}_7$ or $\mathbf{E}_8$, then $\operatorname{Sym}(\Gamma)=1$
     \item if $S$ is $\mathbf{D}_4$, then $\operatorname{Sym}(\Gamma)=\mathbb{S}_3$
     \item for all other types, $\operatorname{Sym}(\Gamma)=\mathbb{Z}_2$
\end{enumerate}
\end{lemma}
If S is $\mathbf{A}_p$, $p\geq2$, $\mathbf{D}_{2k+1}$ or $\mathbf{E}_8$, then the only nontrivial symmetry of $\Gamma$ induces $-\operatorname{id}$ on $\mathcal{S}$. If $S$ is $\mathbf{E}_8$ then $\mathcal{S}=0$ and if $S$ is $\mathbf{A}_1$, $\mathbf{A}_7$ of $\mathbf{D}_{2k}$, the groups $\mathcal{S}$ are $\mathbb{F}_2$ modules  and $-\operatorname{id}=\operatorname{id}$ on $\Aut \mathcal{S}$.

Further details on irreducible root systems are found in N. Bourbaki~\cite{Bour}.

\section{Simple quartics}\label{quartics}
\subsection{Quartics and $K3$-surfaces}\label{QuarticsandK3}
A \textit{quartic} is a surface $X\subset\mathbb{P}^3$ of degree four. A quartic is \textit{simple} if all its singular points are simple, i.e., those of type $\mathbf{A},\mathbf{D},\mathbf{E}$. Isomorphism classes of simple singularities are known to be in a one-to-one correspondence with those of irreducible root systems (see Dufree \cite{Dufree} for details). Hence, a set of simple singularities can be identified with a root system, the irreducible summands of the latter (see \S\ref{root.systems}) correspond to the individual singularity points.

Let $X\subset\mathbb{P}^3$ be a simple quartic and consider its minimal resolution of singularities $\tilde{X}$. It is well known that $\tilde{X}$ is a $K3$-surface; hence, $H_2(\tilde{X})\cong2\mathbf{E}_8\oplus3\mathbf{U}$, where $\mathbf{U}$ is the hyperbolic plane defined as $\mathbf{U}:=\mathbb{Z}u_1\oplus\mathbb{Z}u_2$, $u_1^2=u_2^2=0$ and $u_1\cdot u_2=1$. Note that $2\mathbf{E}_8\oplus3\mathbf{U}$ is the only even unimodular lattice of signature $(\sigma_+,\sigma_-)=(3,19)$. We fix the notation $\mathbf{L}_X:=H_2(\tilde{X})$ and  $\mathbf{L}:=2\mathbf{E}_8\oplus3\mathbf{U}$.

For each simple singular point $p$ of $X$ the components of the exceptional divisor are smooth rational $(-2)$-curves spanning a root lattice in $\mathbf{L}_X$. These sublattices are obviously orthogonal and their orthogonal sum, identified with the set of singularities of $X$, is denoted by $\mathbf{S}_X$. The rank $\operatorname{rk}(\mathbf{S}_X)$ equals the total Milnor number $\mu(X)$. Since $\sigma_-(\mathbf{L})=19$ and $\mathbf{S}_X\subset \mathbf{L}$ is negative definite, one has $\mu(X)\leq 19$ (see \cite{Urabe2}, \emph{cf.}, \cite{Persson}). If $\mu(X)=19$, the quartic is called  \textit{maximizing}. We introduce the following objects:
\begin{itemize}
  \item $\s_X\subset \mathbf{L}_X$: the sublattice generated the set of classes of exceptional divisors contracted by the blow-up map $\tilde{X}\rightarrow X$;
  \item $h_X\in \mathbf{L}_X$: the class of the pull-back of a generic plane section of $X$;
  \item $\s_{X,h}=\s_X\oplus \mathbb{Z}h_X\subset \mathbf{L}_X$;
  \item $\tilde{\s}_X \subset \tilde{\s}_{X,h}\subset \mathbf{L}_X$: the primitive hulls of $\s_X$ and $\s_{X,h}$, respectively, \textit{i.e}, $\tilde{\s}_X:=(\s_X\otimes\mathbb{Q})\cap \mathbf{L}_X$ and  $\tilde{\s}_{X,h}:=(\s_{X,h}\otimes\mathbb{Q})\cap \mathbf{L}_X$, .
  \item $\omega_X \subset \mathbf{L}_X\otimes \mathbb{R}$: the oriented $2$-subspace spanned by the real and imaginary parts of the class of a holomorphic $2$-form on $\tilde{X}$ (the \emph{period} of $\tilde{X}$).
\end{itemize}
The triple $(\s_X, h_X, \mathbf{L}_X)$ is called the \textit{homological type of} $X$.
\subsection{Abstract homological types}\label{abstract.homological.types}
As explained above the set of singularities of a quartic $X\in\mathbb{P}^3$ can be viewed as a root lattice $\mathbf{S}\subset \mathbf{L}$.
\begin{definition}\label{configuration}
A \textit{configuration} (extending a given set of singularities $\mathbf{S}$) is a finite index extension $\tilde{\mathbf{S}}_h\supset \mathbf{S}_h:=\s\oplus\mathbb{Z}h$, $h^2=4$, satisfying the following conditions:
\begin{enumerate}
\item each root  $r\in (\mathbf{S}\otimes\mathbb{Q})\cap\tilde{\mathbf{S}}_h$ with $r^2=-2$ is in $\mathbf{S}$,
\item $\tilde{\mathbf{S}}_h$ does not contain an element $v$ with $v^2=0$ and $v\cdot h=2$.
\end{enumerate}
An \emph{automorphism} of a configuration $\tilde{\s}_h$ is an auto-isometry of $\tilde{\s}_h$ preserving $h$. The group of automorphisms of $\tilde{\s}_h$ is denoted by $\Aut_h(\tilde{\mathbf{S}}_h)$. One has the obvious inclusions $\Aut_h(\tilde{\mathbf{S}}_h)\subset O(\tilde{\mathbf{S}})\subset O(\s)$, the latter is due to $(1)$ in  Definition \ref{configuration}, since $\s$ is recovered as the sublattice in $h^{\bot}\subset\tilde{\s}_h$ generated by roots.
\end{definition}
\begin{definition}\label{sinhom}
An \textit{abstract homological type} extending a fixed set of singularities $\mathbf{S}$ is an extension of $\mathbf{S}_h:=\mathbf{S}\oplus\mathbb{Z}h$, $h^2=4$, to a lattice $L$ isomorphic to $2\mathbf{E}_8\oplus3\mathbf{U}$, such that the primitive hull $\tilde{\mathbf{S}}_h$ of $\mathbf{S}_h$ in $L$ is  a configuration.
\end{definition}
An abstract homological type is uniquely determined by the triple $\mathcal{H}=(\mathbf{S},h,\mathbf{L})$.
An \textit{isomorphism} between two abstract homological types $\mathcal{H}_i=(\mathbf{S}_i,h_i,\mathbf{L}_i)$, $i=1,2$, is an isometry $\mathbf{L}_1\rightarrow \mathbf{L}_2$, taking $h_1$ and $\mathbf{S}_1$ to $h_2$ and $\mathbf{S}_2$, respectively (as a set).

Given an abstract homological type $\mathcal{H}=(\mathbf{S},h,\mathbf{L})$, we let $\tilde{\s}:=(\s\otimes\mathbb{Q})\cap\mathbf{L}$ and $\tilde{\s}_h:=(\s_h\otimes\mathbb{Q})\cap\mathbf{L}$ be the primitive hulls of $\s$ and $\s_h$, respectively. Note that $\tilde\s=h^{\perp}_{\tilde{\s}_h}$, \emph{i.e.}, $\tilde\s$ is also the primitive hull of $h^{\perp}$. The orthogonal complement $\mathbf{S}_h^\bot$ is a non-degenerate lattice with $\sigma_+\s^\bot=2$. It follows that all positive definite $2$-subspaces in $\mathbf{S}_h^\bot\otimes\mathbb{R}$ can be oriented in a coherent way (see \S \ref{positive.sign.structure}).
\begin{definition}
An \textit{orientation} of an abstract homological type $\mathcal{H}=(\mathbf{S},h,\mathbf{L})$ is a choice $\theta$ of one of the coherent orientations of positive definite $2$-subspaces of $\mathbf{S}_h^\bot\otimes\mathbb{R}$
\end{definition}
An \textit{isomorphism} between two oriented singular homological type $(\mathcal{H}_i,\theta_i)$, $i=1,2$, is an isomorphism $\mathcal{H}_1\rightarrow\mathcal{H}_2$, taking $\theta_1$ to $\theta_2$. A singular homological type is called \textit{symmetric} if $(\mathcal{H},\theta)$ is isomorphic to $(\mathcal{H},-\theta)$ for some orientation $\theta$ of $\mathcal{H}$, i.e., $\mathcal{H}$ admits an automorphism reversing the orientation.

\subsection{Classification of singular quartics}\label{classquartics}
Due to Saint-Donat~\cite{Donat} and Urabe~\cite{Urabe2}, a triple $\mathcal{H}=(\mathbf{S},h,\mathbf{L})$ is isomorphic to the homological type  $(\mathbf{S}_X,h_X,\mathbf{L}_X)$ of a simple quartic $X\subset \mathbb{P}^3$ if and only if $\mathcal{H}$ is an abstract homological type in the sense of Definition $\ref{sinhom}$. In this case, the oriented $2$-subspace $\omega_X$ introduced in $\S3.1$ defines an orientation of $\mathcal{H}$.
\begin{theorem}[see Theorem 2.3.1 in \cite{AI}]\label{def.class}
The map sending a simple quartic surface $X\subset\mathbb{P}^3$ to its oriented homological type establishes a one to one correspondence between the set of equisingular deformation classes of quartics with a given set of simple singularities $\mathbf{S}$ and the set of isomorphism classes of oriented abstract homological types extending $\mathbf{S}$. Complex conjugate quartics have isomorphic homological types that differ by the orientations.
\end{theorem}
\begin{definition}\label{nonspecial}
  A quartic $X$ is called \emph{non-special} if its homological type is primitive, \textit{i.e.}, $\s_h\subset \mathbf{L}$ is a primitive sublattice.
\end{definition}
Note that the homological type $\mathcal{H}=(\mathbf{S},h,\mathbf{L})$ is primitive if and only if $\tilde{\s}_h=\s_h$, in this case, one has $\operatorname{disc} \tilde{\s}_h=\mathbf{\mathcal{S}}\oplus\langle\frac{1}{4}\rangle$ and $\Aut_h(\tilde{\mathbf{S}}_h)= O(\mathbf{S})$.

For a given set of simple singularities $\s$, the corresponding equisingular stratum of quartics is denoted by $\mathcal{M}(\s)$. Our primary interest is the family $\mathcal{M}_1(\s)\subset \mathcal{M(\s)}$ constituted by the non-special quartics with the set of singularities $\s$. More generally, since the kernel $\mathcal{K}$ of the finite index extension $\s_h\subset\tilde{\s}_h$ is obviously invariant under equisingular deformations, one can consider the strata $\mathcal{M}_{\ast}(\s)\subset\mathcal{M}(\s)$ where the subscript $\ast$ is the sequence of invariant factors of the kernel $\mathcal{K}$.

\section{Proofs}

\subsection{Proof of Theorem \ref{motivation}}\label{proof.motivation}
Note that $X\smallsetminus(\operatorname{Sing}X\cup H)\cong \tilde{X}\smallsetminus(E\cup H)$, where $\tilde{X}$ is the minimal resolution of $X$ and $E$ is the exceptional divisor of the blow up $\tilde{X}\rightarrow X$. Recall that $\s_X$ is the sublattice in $\mathbf{L}_X=H_2(\tilde{X})$ generated by the components of $E$ (see \S \ref{QuarticsandK3}). Thus, one has $H_2(E\cup H)=\s_X\oplus\mathbb{Z}h_X\subset\mathbf{L}_X=H_2(\tilde{X})$.

We have the following cohomology exact sequence of pair $(\tilde X, E\cup H)$:
\begin{align*}
    \cdots\xrightarrow{j^{\ast}} H^2(\tilde{X})\xrightarrow{i^{\ast}}H^2(E\cup H)\xrightarrow{\delta}H^3(\tilde{X},E\cup H)\xrightarrow{j^{\ast}}\underbrace{H^3(\tilde{X})}_{0}\rightarrow\cdots.
\end{align*}
Hence, $H^3(\tilde{X},E\cup H)=\operatorname{coker}i^{\ast}$. By universal coefficients, since all groups involved are free, $i^{\ast}$ is the adjoint of the map
$$i_{\ast}:H_2(E\cup H)\rightarrow H_2(X),$$
which is the inclusion $\s_{X,h}\hookrightarrow \mathbf{L}_X$. Thus, we have an exact sequence
\begin{align*}
   0\rightarrow H_2(\tilde{X},E\cup H)\xrightarrow{i_{\ast}}H_2(\tilde{X})\rightarrow\faktor{{\tilde{\s}_{X,h}}}{\s_{X,h}}\oplus \mathrm{F}\rightarrow 0,
\end{align*}
where $\mathrm{F}$ is a finitely generated free abelian group. This sequence can be regarded as a free resolution of $\faktor{{\tilde{\s}_{X,h}}}{\s_{X,h}}\oplus \mathrm{F}$ and, by the definition of derived functor, we have the following isomorphisms
\begin{align*}
   \operatorname{coker}i^{\ast}= \operatorname{Ext}\Big(\faktor{\tilde{\s}_{X,h}}{\s_{X,h}}\oplus\mathrm{F},\mathbb{Z}\Big)=
   \operatorname{Ext}\Big(\faktor{\tilde{\s}_{X,h}}{\s_{X,h}},\mathbb{Z}\Big).
\end{align*}
Combining these observations with Poincar\'{e}--Lefschetz duality $H_1(\tilde{X}\smallsetminus (E\cup H))=H^3(\tilde{X}, E\cup H)$, we conclude that
$$H_1(\tilde{X}\smallsetminus (E\cup H))= \operatorname{Ext}\Big(\faktor{\tilde{\s}_{X,h}}{\s_{X,h}},\mathbb{Z}\Big)\cong\faktor{\tilde{\s}_{X,h}}{\s_{X,h}}$$
(the last isomorphism being not natural). In particular $H_1(\tilde{X}\smallsetminus (E\cup H))=0$ if and only if $\s_{X,h}=\tilde{\s}_{X,h}$ \emph{i.e}, if and only if $X$ is non-special.\qed

\subsection{Proof of Theorem \ref{maintheorem}}\label{proof.maintheorem}

For the reader's convenience, we divide the proof into three propositions; Theorem \ref{maintheorem} is their immediate consequence.
\begin{proposition}\label{propexistance}
Realizable are all sets of singularities that can be obtained by a perturbation from either the $59$ maximizing sets of  singularities listed in Table $1$ or $19$ sets of singularities with the Milnor number $18$ listed in Table $2$.
\end{proposition}
\begin{proof}
 According to Theorems \ref{def.class} and Definition \ref{nonspecial}, a set of singularities $\s$ is realized by a non-special quartic if and only if $\s$ extends to a primitive homological type. Thus, we are interested in primitive extensions $\s_h\hookrightarrow\mathbf{L}=3\mathbf{U}\oplus3\mathbf{E}_8$. Since the homological type is primitive, one has $\operatorname{disc}\tilde{\s}_h =\mathcal{S}\oplus\langle\frac{1}{4}\rangle$, and the realizable sets are easily found by using Nikulin's Existence Theorem (Theorem \ref{Existence}) applied to the genus of the transcendental lattice $\mathbf{T}:=\s^{\perp}$, which is determined by $\s$, see \S\ref{lattice.extensions}. Implementing the algorithm in GAP \cite{GAP}, we found that $2872$ sets of simple singularities are realized by non-maximal non-special quartics and $59$ sets of simple singularities are realized by maximal non-special quartics. According to E. Looijenga \cite{Looi}, deformation classes of perturbations of an individual simple singularity of type $ \textbf{S}$ are in a one-to-one correspondence with the  isomorphism classes of primitive extensions $\textbf{S}'\hookrightarrow \textbf{S}$ of root lattices, see \S \ref{root.systems} and \S\ref{lattice.extensions}. As shown in \cite{Dynkin}, the latter is the case if and only if the Dynkin graph of $\textbf{S}'$ is an induced subgraph of that of $\textbf{S}$. Hence, given a simple quartic $X$, any perturbation $X$ to a simple quartic $X'$ gives rises to a perturbation of the set of singularities $\textbf{S}$ of $X$ to the set of singularities $\textbf{S}'$ of $X'$. Conversely, any induced subgraph of the Dynkin graph of a simple quartic $X$ is that of an appropriate small perturbation $X'$ of $X$. Proof of this statement repeats, almost literally, the proof of a similar theorem for plane sextic curves (see Proposition $5.1.1 $ in \cite{Alex.irr.sextics}). Accordingly, the list of $2872$ sets of simple singularities realized by non-maximal non-special quartics is compared against the list of all perturbations of the $59$ maximizing sets of singularities given in Table $1$ and $19$ sets of singularities with Milnor number $18$ given in Table $2$. The two lists coincide.
\end{proof}
Let $\mathbf{S}$ be one of the realizable sets of singularities  and $\mathbf{T}$ a representative of the genus $g(\mathbf{S}^{\bot}_h)$. By Theorem \ref{def.class}, the connected components of the space $\mathcal{M}_1(\mathbf{S})$ modulo complex conjugation $\operatorname{conj} : \mathbb{P}^3\rightarrow \mathbb{P}^3 $ are enumerated by the isomorphism classes of primitive homological types extending $\mathbf{S}$. We investigate these isomorphism classes separately for the maximizing case, \textit{i.e.}, $\mu(\mathbf{S})=19$, and non-maximizing case, \textit{i.e}, $\mu(\mathbf{S})\leq 18$.

If $\mu(\mathbf{S})=19$,  the transcendental lattice $ \mathbf{T}$ is a positive definite sublattice of rank $2$, and the numbers $(r,c)$ of connected components of the space $\mathcal{M}_1(\mathbf{S})$ listed in Table $1$ can easily be computed by Gauss theory of binary quadratic forms \cite{Gauss} (A. Degtyarev, private communication); details will appear elsewhere. Thus, throughout the rest of the proof we assume $\mu(\mathbf{S})\leq 18$.
\begin{proposition}\label{propM1Sconjconnected}
For each realizable set of singularities $\mathbf{S}$ with $\mu(\mathbf{S})\leq 18$, the space  $\mathcal{M}_1(\mathbf{S})/\operatorname{conj}$ is connected.
\end{proposition}
\begin{proof}
If $\mu(\mathbf{S})\leq 18$, then $\mathbf{T}$ is an indefinite lattice with $\operatorname{rk}\mathbf{T}\geq 3$ and we can apply Mirranda--Morison's theory. We try to enumerate primitive homological types $\mathcal{H}=(\s,h,\mathbf{L})$ extending $\s$, \textit{i.e.}, the primitive extensions $\mathbf{S}_h\hookrightarrow L$. Since the extension is primitive, $\tilde{\mathbf{S}}_h=\mathbf{S}_h$, one has $\operatorname{disc}\tilde{\mathbf{S}}_h=\mathbf{\mathcal{S}}\oplus\langle\frac{1}{4}\rangle$ and $\Aut_h(\mathbf{S}_h)\cong O(\mathbf{S})$. Then we have a well-defined homomorphism $d^{\bot}\colon O(\s)\rightarrow E(\mathbf{T})$, and by Corollary \ref{coker},
\begin{align}\label{pi0}
    \pi_0(\mathcal{M}_1(\mathbf{S})/\operatorname{conj})\cong \operatorname{Coker}\big(d^{\bot} :O(\mathbf{S})\rightarrow E(\mathbf{T})\big).
\end{align}
Thus, the space $\mathcal{M}_1(\mathbf{S})/\operatorname{conj}$ is connected (equivalently, the primitive homological type extending $\mathbf{S}$ is unique up to isomorphism) if and only if the map $d^{\bot}$ is surjective, and it is this latter statement that we prove below.

Out of the $2872$ sets of singularities realized by non-special non-maximizing quartics, for $2830$ sets of singularities one  gets $E(\mathbf{T})=1$ by using (\ref{orderofEN}), and the assertion follows automatically.

For the remaining $42$ cases, one has $|E(\mathbf{T})|\neq1$.
Among these, there are $18$ set of singularities containing a point of type $\mathbf{A}_4$ and satisfying the hypothesis of Lemma \ref{1irr} or Corollary \ref{c2irr} with $p=5$. For these set of singularities one has $|E(\mathbf{T})|=2$ and a nontrivial symmetry of any type $\mathbf{A}_4$ point maps to the generator $-1\in E(\mathbf{T})$.

There are $8$ sets of singularities containing a point of type $\mathbf{A}_2$ and satisfying the hypothesis of Lemma \ref{2irr} with $p=2$, $q=3$. For these $8$ cases, one has $|E(\mathbf{T})|=2$ and a nontrivial symmetry of any type $\mathbf{A}_2$ point maps to the generator $-1\in E(\mathbf{T})$. For the following $4$ sets of singularities,
\begin{align*}
    \mathbf{D}_9\oplus\mathbf{A}_3\oplus3\mathbf{A}_2,\quad \mathbf{D}_7\oplus\mathbf{D}_5\oplus3\mathbf{A}_2,\\
   \mathbf{A}_{11}\oplus\mathbf{A}_3\oplus2\mathbf{A}_2,\quad \mathbf{A}_8\oplus2\mathbf{A}_3\oplus2\mathbf{A}_2,
\end{align*}
one has $|E(\mathbf{T})|=4$. Each of these $4$ sets  has two irregular primes $p=2$, $q=3$, and for all of them the homomorphism given by Lemma \ref{2irr} is
\begin{align*}
     \mathrm{e}(r_\alpha)=(\delta_2(\alpha)\cdot\delta_3(\alpha), |\alpha|_2\cdot|\alpha|_3)\in \{\pm1\}\times\{\pm1\}.
\end{align*}
A symmetry of any type $\mathbf{A}_2$ point and a transposition $\mathbf{A}_2\leftrightarrow\mathbf{A}_2$ give rise to reflections $r_{\alpha},r_{\sigma}\in \mathcal{T}$ with $\alpha^2=\frac{2}{3}$ and $(\sigma)^2=\frac{4}{3}$. The images  $\mathrm{e}(r_{\alpha})=(-1,-1)$ and $\mathrm{e}(r_{\sigma})=(-1,1)$ are linearly independent, thus generating the group $E(\mathbf{T})$.

The $9$ sets of singularities listed in Table $3$ still satisfy the assumptions of Lemma \ref{2irr}, which yields $|E(\mathbf{T})|=2$. Also shown in the table are the irregular primes $(p,q)$, the homomorphism $\mathrm{e}\colon\Aut (\mathcal{T})\rightarrow E(\mathbf{T})$, and an automorphism of $\s$ generating $E(\mathbf{T})$.
%
\begin{table}\label{table.of.map.e}
\caption{\small Extremal singularities }
\begin{tabular}{|l|c|l|l|}
  \hline
  Singularities &$(p,q)$ &$\mathrm{e}\colon\Aut(\mathcal{T})\rightarrow E(\mathbf{T})$& generators of $E(\mathbf{T})$ \\ \hline
 $ \mathbf{E}_8\oplus2\mathbf{A}_3\oplus2\mathbf{A}_2$ & $(2,3)$ & $\mathrm{e}(r_{\alpha})=\delta_2(\alpha)\cdot\delta_3(\alpha)\cdot|\alpha|_2\cdot|\alpha|_3$ &   $\mathbf{A}_2\leftrightarrow\mathbf{A}_2$  \\
$ 2\mathbf{E}_6\oplus2\mathbf{A}_3$  & $(2,3)$ &$\mathrm{e}(r_{\alpha})=\delta_2(\alpha)\cdot\delta_3(\alpha)\cdot|\alpha|_2\cdot|\alpha|_3$ & symmetry of $\mathbf{A}_3$ \\
$\mathbf{D}_{11}\oplus \mathbf{A}_3\oplus2\mathbf{A}_2$  &$ (2,3)$ & $\mathrm{e}(r_{\alpha})=\delta_2(\alpha)\cdot\delta_3(\alpha)\cdot|\alpha|_2\cdot|\alpha|_3$ & $\mathbf{A}_2\leftrightarrow\mathbf{A}_2$\\
$2\mathbf{D}_7 \oplus2\mathbf{A}_2$ & $(2,3)$ & $\mathrm{e}(r_{\alpha})=\delta_2(\alpha)\cdot\delta_3(\alpha)\cdot|\alpha|_2\cdot|\alpha|_3$ & $\mathbf{A}_2\leftrightarrow\mathbf{A}_2$  \\
 $ 2\mathbf{D}_5\oplus2\mathbf{A}_4$ & $(2,5)$ & $\mathrm{e}(r_{\alpha})=\delta_2(\alpha)\cdot\delta_5(\alpha)$ & $\mathbf{A}_4\leftrightarrow\mathbf{A}_4$  \\
 $ \mathbf{D}_5\oplus\mathbf{A}_6\oplus\mathbf{A}_3\oplus2\mathbf{A}_2$ & $(2,3)$ &$\mathrm{e}(r_{\alpha})=\delta_2(\alpha)\cdot\delta_3(\alpha)\cdot|\alpha|_2\cdot|\alpha|_3$ &   $\mathbf{A}_2\leftrightarrow\mathbf{A}_2$  \\
$ \mathbf{A}_7\oplus\mathbf{A}_4\oplus\mathbf{A}_3\oplus2\mathbf{A}_2$ & $(2,3)$ & $\mathrm{e}(r_{\alpha})=\delta_2(\alpha)\cdot\delta_3(\alpha)\cdot|\alpha|_2\cdot|\alpha|_3$ &  $\mathbf{A}_2\leftrightarrow\mathbf{A}_2$  \\
$ 2\mathbf{A}_6\oplus2\mathbf{A}_3$ &$( 2,7)$ & $\mathrm{e}(r_{\alpha})=|\alpha|_2\cdot|\alpha|_7$ &  symmetry of $\mathbf{A}_3$ \\
$ 2\mathbf{A}_6\oplus3\mathbf{A}_2$ & $(3,7)$ & $\mathrm{e}(r_{\alpha})=\delta_3(\alpha)\cdot\delta_7(\alpha)|\cdot\alpha|_3\cdot|\alpha|_7$ & $\mathbf{A}_2\leftrightarrow\mathbf{A}_2$  \\

\hline
\end{tabular}

\end{table}

Finally, what remains are the three sets of singularities
\begin{align*}
     \mathbf{D}_4\oplus2\mathbf{A}_4\oplus3\mathbf{A}_2,\quad 2\mathbf{A}_7\oplus2\mathbf{A}_2\quad2\mathbf{A}_4\oplus 2\mathbf{A}_3\oplus2\mathbf{A}_2,
\end{align*}
to which Lemmas \ref{1irr}, \ref{2irr} or Corollary \ref{c2irr} do not apply. For them, we compute the group $E(\mathbf{T})$ directly from the definition (\ref{E(N)}) which can be restated as
\begin{align*}
    E(T)=\prod_{p|\operatorname{det}(T)}\Gamma_{p,0}\Big/\prod_{p|\operatorname{det}(T)}\Sigma_p^{\sharp}(T)\cdot\varphi(\Gamma _0),
\end{align*}
where we identify the inclusion $\Gamma_0\hookrightarrow\Gamma_{\mathbb{A},0}$ with the product $\varphi:=\prod_p\varphi_p$ (see \ref{fayp}).
For example, for the case
\begin{align*}
    2\mathbf{A}_4\oplus 2\mathbf{A}_3\oplus2\mathbf{A}_2,
\end{align*}
the computation can be summarized in the following table:
\begin{center}
\begin{tabular}{c|cc|cc|cc|}

 &\multicolumn{2}{|c|}{$\Gamma_{5,0}$}&\multicolumn{2}{|c|}{$\Gamma_{3,0}$}&\multicolumn{2}{|c|}{$\Gamma{'}_{2,0}$}\\\hline
 generator of $\Sigma_5^{\sharp}(T)$&-1 & -1 &1 & 1 & 1 & 1  \\\hline
generator of $\Sigma_3^{\sharp}(T) $& 1 & 1 & -1 & 1 & 1 & 1 \\\hline
generator of $\Sigma_2^{\sharp}(T)$  &1 & 1 & 1 & 1 & 1 & 1 \\
  \hline
$\varphi(-1,1) $& -1 & 1 & -1 & 1 & -1 & 1 \\\hline
 $\varphi(1,-1)$&1 & 1 & 1 & -1 & 1 & -1  \\
  \hline\cline{1-7}\cline{1-7}
 & \multicolumn{1}{|c|}{$\delta_5$} &  \multicolumn{1}{|c|}{$|\cdot|_5$} &  \multicolumn{1}{|c|}{$\delta_3$} &  \multicolumn{1}{|c|}{$|\cdot|_3$} &  \multicolumn{1}{|c|}{$\delta_2$} &  \multicolumn{1}{|c|}{$|\cdot|_2$}  \\\hline

 a symmetry of $\mathbf{A}_4$&\multicolumn{1}{|c|}{-1} &  \multicolumn{1}{|c|}{-1} &  \multicolumn{1}{|c|}{1} &  \multicolumn{1}{|c|}{-1} &  \multicolumn{1}{|c|}{1} &  \multicolumn{1}{|c|}{1}\\\hline

a transposition $\mathbf{A}_4\leftrightarrow \mathbf{A}_4$&\multicolumn{1}{|c|}{-1} &  \multicolumn{1}{|c|}{1} &  \multicolumn{1}{|c|}{1} &  \multicolumn{1}{|c|}{-1} &  \multicolumn{1}{|c|}{1} &  \multicolumn{1}{|c|}{1} \\\hline

a symmetry of $\mathbf{A}_2$&\multicolumn{1}{|c|}{1} &  \multicolumn{1}{|c|}{-1} &  \multicolumn{1}{|c|}{-1} &  \multicolumn{1}{|c|}{1} &  \multicolumn{1}{|c|}{1} &  \multicolumn{1}{|c|}{-1}\\\hline

 a transposition $\mathbf{A}_2\leftrightarrow \mathbf{A}_2$&\multicolumn{1}{|c|}{1} &  \multicolumn{1}{|c|}{-1} &  \multicolumn{1}{|c|}{-1} &  \multicolumn{1}{|c|}{-1} &  \multicolumn{1}{|c|}{1} &  \multicolumn{1}{|c|}{-1} \\\hline

\end{tabular}
\end{center}
The rank of the matrix composed by the $9$ rows of the table (see, Remark \ref{remark}) is $6=\operatorname{dim}\Gamma'_{2,0}+\operatorname{dim}\Gamma_{3,0}+\operatorname{dim}\Gamma_{5,0}$, which implies that $d^{\bot}$ is surjective. For the remaining two cases
\begin{align*}
 \mathbf{D}_4\oplus2\mathbf{A}_4\oplus3\mathbf{A}_2,\quad2\mathbf{A}_7\oplus2\mathbf{A}_2,
\end{align*}
the computation is almost literally the same. For $ 2\mathbf{A}_7\oplus2\mathbf{A}_2$, where $\Sigma_2^{\sharp}(\mathbf{T})\not\supset\Gamma_{2,2}$, we have to modify $|\cdot|_2$ by replacing $\chi_2$ with $\chi_2(u)= u \mod8\in\{1,3,5,7\}=\mathbb{Z}_2^\times/(\mathbb{Z}_2^\times)^2$ and consider the full group $\Gamma_{2,0}$ instead of $\Gamma'_{2.0}$.
\end{proof}
\begin{remark}\label{remark}
Here and below, when speaking about ranks and dimensions, we regard all groups $\Gamma_*$, $E$, $E^+$, etc.\ as $\mathbb{F}_2$-vector spaces. In particular, when computing the rank of a matrix, we need to switch from the multiplicative notation $\{1,-1\}$ to the additive $\{0,1\}$.
\end{remark}
\begin{corollary}[of the proof]
For all sets of singularities $\mathbf{S}$ with $\mu(\mathbf{S})\leq18$, the corresponding transcendental lattice $\mathbf{T}$ is unique in its genus, \textit{i.e.}, $g(\mathbf{T})=1$.
\end{corollary}
\begin{proposition}
If $\mathbf{S}$ is one of
\begin{align*}
  \mathbf{D}_6\oplus2\mathbf{A}_6,\quad\mathbf{D}_5\oplus2\mathbf{A}_6\oplus\mathbf{A}_1,\quad2\mathbf{A}_7\oplus2\mathbf{A}_2,
  \quad3\mathbf{A}_6,\quad2\mathbf{A}_6\oplus2\mathbf{A}_3\,
  \end{align*}
 then $\mathcal{M}_1(\mathbf{S})$ consists of two complex conjugate components; in all other cases with $\mu(\mathbf{S})\leq18$, the stratum $\mathcal{M}_1(\mathbf{S})$ is connected.
\end{proposition}
\begin{proof}
By Proposition \ref{propM1Sconjconnected}, $\mathcal{M}_1(\mathbf{S})$ is connected if and only if the (unique) homological type extending $\mathbf{S}$ is symmetric; otherwise  $\mathcal{M}_1(\mathbf{S})$ consists of two complex conjugate components. By Theorem \ref{Nikulin.homological}, homological type is symmetric if and only if there is an isometry $a\in O(\mathbf{T})$ with $det_+(a)=-1$ satisfying $d(a)\in d^{\psi}(O(\mathbf{S}))$, where $d^{\psi}$ is the map induced by any anti-isometry $\psi\colon\mathbf{\mathcal{S}}\oplus\langle\frac{1}{4}\rangle\rightarrow\mathbf{\mathcal{T}}$. We consider separately the cases $|E(\mathbf{T})|=|E^+(\mathbf{T})|$ and $|E(\mathbf{T})|<|E^+(\mathbf{T})|$.
\begin{lemma}\label{E=E+}
If $|E(\mathbf{T})|=|E^+(\mathbf{T})|$, then $\mathcal{M}_1(\mathbf{S})$ is connected.
\end{lemma}
\begin{proof}
By definition, we have an exact sequence
\begin{align}\label{ExactsquenceEE+}
     0\rightarrow \Gamma_0/\Gamma_0^{--}\cdot\tilde{\Sigma}(\mathbf{T})\rightarrow E^+(\mathbf{T})\rightarrow E(\mathbf{T})\rightarrow 0.
\end{align}
 Hence, $|E(\mathbf{T})|=|E^+(\mathbf{T})|$ if and only if $\tilde{\Sigma}(\mathbf{T})\not\subset\Gamma_0^{--}$. Then, by Proposition \ref{propdet+}, there exist a $+$-disorienting isometry of $\mathbf{T}$ inducing the identity on $\operatorname{disc}\mathbf{T}$, and Theorem \ref{Nikulin.homological} applies.
\end{proof}
\begin{lemma}\label{dt+.surjective}
If $|E(\mathbf{T})|<|E^+(\mathbf{T})|$, then $\mathcal{M}_1(\mathbf{S})$ is connected if and only if $d^{\bot}_+\colon O(S)\rightarrow E^+(\mathbf{T})$ is an epimorphism.
\end{lemma}
\begin{proof}
The  non-trivial element of the kernel $K:=\Gamma_0/\Gamma_0^{--}\cdot\tilde{\Sigma}(\mathbf{T})\cong\{\pm1\}$ in (\ref{ExactsquenceEE+}) is the image under the composed map $O(\mathbf{T})\rightarrow\Aut(\mathbf{\mathcal{T}})\rightarrow E^+(\mathbf{T})$ of any element $a\in O(\mathbf{T})$ with $\operatorname{det}_+(a)=-1$. Thus, $\mathcal{M}_1(\mathbf{S})$ is connected if and only if $\operatorname{Im}(d_+^{\bot})\cap K\neq 0 $, \textit{i.,e.}, $\operatorname{rank}d_+^{\bot}>\operatorname{rank}d^{\bot}$ (see Remark \ref{remark}). On the other hand, Proposition \ref{propM1Sconjconnected} can be recast in the form $ \operatorname{rank}d^{\bot}=\operatorname{dim}E(\mathbf{T})$. Since $\operatorname{dim}E_+(\mathbf{T})=\operatorname{dim}E(\mathbf{T})+1$, the statement follows.
\end{proof}

Lemma \ref{E=E+} implies the connectedness of $\mathcal{M}_1(\mathbf{S})$ for $2721$ sets of singularities.
For the remaining $151$ sets of singularities, one has $|E(\mathbf{T})|<|E^+(\mathbf{T})|$. For $118$ of them, one has $|E(\mathbf{T})|=1$ and $\tilde{\Sigma}_p(\mathbf{T})\subset\Gamma_0^{--}$ for some prime $p$. Since $|E(\mathbf{T})|=1$, the map $d:O(\mathbf{T})\rightarrow\Aut(\mathbf{\mathcal{T}})$ is surjective and the isomorphism
\begin{align*}
    \Aut(\mathcal{\mathbf{\mathcal{T}}})/O^+(\mathbf{T})=\Gamma_0/\Gamma_0^{--}\cdot\tilde{\Sigma}(\mathbf{T})=E^+(\mathbf{T})=\{\pm1\}
\end{align*}
(see(\ref{ExactsquenceEE+})) is the descent of $\operatorname{det}_+$, which is well-defined due Proposition \ref{propdet+}.
In most cases we can use Lemma \ref{lemmadet+} to show that there exists an element $a\in O(\mathbf{S})$ such that $\operatorname{det}_+(d^{\psi}(a))=-1$. Namely,
\begin{itemize}
  \item if $p=2$, take for $a$ a nontrivial symmetry of $\mathbf{A}_2$, $\mathbf{D}_5$, $\mathbf{E}_6$ or $\mathbf{D}_9$,
  \item if $p=3$, take for $a$ a nontrivial symmetry of $\mathbf{A}_2$, $\mathbf{D}_5$ or $\mathbf{A}_8$,
  \item if $p=7$, take for $a$ a nontrivial symmetry of $\mathbf{A}_2$.
\end{itemize}
The three sets of singularities
\begin{align*}
    \mathbf{D}_6\oplus2\mathbf{A}_6, \quad\mathbf{D}_5\oplus2\mathbf{A}_6\oplus\mathbf{A}_1,\quad3\mathbf{A}_6
\end{align*}
with $p=7$ are exceptional (and are listed as such in the statement), as Lemma \ref{lemmadet+} implies ${\operatorname{det}_+}\circ d^{\psi}\equiv1$. Indeed, the image of $d^{\psi}$ is generated by the reflections $r_{\alpha}$ with either
\begin{itemize}
 \item $\alpha\in\mathbf{\mathcal{T}}_7$, $\alpha^2=\frac{6}{7}$ (a nontrivial symmetry of $\mathbf{A}_6$) or,
  \item $\alpha\in\mathbf{\mathcal{T}}_7$, $\alpha^2=\frac{12}{7}$ (a transposition $\mathbf{A}_6\leftrightarrow\mathbf{A}_6$) or,
 \item $\alpha\in\mathbf{\mathcal{T}}_2$  (a nontrivial symmetry of $\mathbf{D}_6$ or $\mathbf{D}_5$).
\end{itemize}
One has $\delta_7(\alpha)=|\alpha|_7=-1$ in the first two cases and $\delta_7(\alpha)=|\alpha|_7=1$ in the last one.

There are $30$ other sets of singularities still satisfying the condition $\Gamma_{2,2}\subset\Sigma_2^{\sharp}(\mathbf{T})$ and having two irregular primes, so that we can apply Lemma \ref{2irr+}. Among them, $13$ sets of singularities contain two type $\mathbf{A}_2$ points and have $(p,q)=(2,3)$ and $|E^+(\mathbf{T})|=4$. A nontrivial symmetry of any type $\mathbf{A}_2$ point and a transposition $\mathbf{A}_2\leftrightarrow\mathbf{A}_2$ map to two linearly independent elements generating $E^+(\mathbf{T})$. The remaining $17$ sets of singularities are listed in Table $4$, where we indicate the irregular primes $(p,q)$, order of $E^+:=E^+(\mathbf{T})$ and a collection of isometries of $\s$ whose images generate $E^+(\mathbf{T})$.
\begin{table}\label{tableof17}
\caption{\small Extremal singularities }
\begin{tabular}{|l|c|c|l|}
  \hline
Singularities & $(p,q)$ &$|E^+|$ & isometries of $\mathbf{S}$ generating $E^+(\mathbf{T})$ \\ \hline
$2\mathbf{E}_6\oplus2\mathbf{A}_3$& $(2,3)$ &$4$ & symmetry of $\mathbf{E}_6$; $\mathbf{A}_3\leftrightarrow\mathbf{A}_3$\\
$\mathbf{D}_9\oplus \mathbf{A}_3\oplus3\mathbf{A}_2$  & $(2,3)$ &$8$ & symmetries of $\mathbf{A}_2$, $\mathbf{A}_3$; $\mathbf{A}_2\leftrightarrow\mathbf{A}_2$\\
$ \mathbf{D}_8\oplus\mathbf{A}_6\oplus2\mathbf{A}_2$ &$(2,3)$ &$2$&   $\mathbf{A}_2\leftrightarrow\mathbf{A}_2$  \\
$ \mathbf{D}_{8}\oplus\mathbf{A}_3\oplus3\mathbf{A}_2$  & $(2,3)$ &$2 $& $\mathbf{A}_2\leftrightarrow\mathbf{A}_2$\\
$\mathbf{D}_7\oplus\mathbf{D}_5\oplus3\mathbf{A}_2$  & $(2,3)$ &$8$ & symmetries of $\mathbf{A}_2$, $\mathbf{D}_5$; $\mathbf{A}_2\leftrightarrow\mathbf{A}_2$\\
$ \mathbf{D}_7\oplus\mathbf{D}_4\oplus3\mathbf{A}_2$  & $(2,3)$ &$2$ & $\mathbf{A}_2\leftrightarrow\mathbf{A}_2$ \\
$\mathbf{D}_6 \oplus2\mathbf{A}_4\oplus2\mathbf{A}_2$ & $(3,5)$ & $2 $&  symmetry of $\mathbf{A}_2$ \\
$2\mathbf{D}_5\oplus 2\mathbf{A}_4$& $(2,5)$ & $4 $&   symmetry of $\mathbf{D}_5$; $\mathbf{A}_4\leftrightarrow\mathbf{A}_4$  \\
$\mathbf{D}_5 \oplus2\mathbf{A}_4\oplus2\mathbf{A}_2\oplus\mathbf{A}_1$ &$(3,5)$ & $4$ &  symmetry of $\mathbf{A}_2$; $\mathbf{A}_2\leftrightarrow\mathbf{A}_2$  \\
$ \mathbf{D}_4\oplus2\mathbf{A}_6\oplus\mathbf{A}_2$ & $(2,7)$ &$2$ &    symmetry of $\mathbf{A}_6$  \\
$\mathbf{A}_{11}\oplus\mathbf{A}_3\oplus2\mathbf{A}_2$  & $(2,3)$ &$8$ & symmetries of $\mathbf{A}_2$, $\mathbf{A}_3$; $\mathbf{A}_2\leftrightarrow\mathbf{A}_2$\\
$ \mathbf{A}_8\oplus2\mathbf{A}_3\oplus2\mathbf{A}_2$  & $(2,3)$ &$8$& symmetries of $\mathbf{A}_2$, $\mathbf{A}_3$; $\mathbf{A}_2\leftrightarrow\mathbf{A}_2$\\
$ 2\mathbf{A}_6\oplus2\mathbf{A}_3$ &$(2,7)$ & $4$ & exceptional\\
$ 2\mathbf{A}_6\oplus3\mathbf{A}_2$&$(3,7)$ & $4$ &  symmetries of $\mathbf{A}_2$, $\mathbf{A}_6$ \\
$ 2\mathbf{A}_5\oplus2\mathbf{A}_4$& $(3,5)$ & $2$&  symmetry of $\mathbf{A}_4$ \\
$ 3\mathbf{A}_4\oplus2\mathbf{A}_2\oplus2\mathbf{A}_1$ & $(3,5)$ & $4$ &  symmetry of $\mathbf{A}_2$; $\mathbf{A}_2\leftrightarrow\mathbf{A}_2$  \\
$ 2\mathbf{A}_4\oplus4\mathbf{A}_2$ &$(2,3)$ & $2$ &  $\mathbf{A}_2\leftrightarrow\mathbf{A}_2$  \\
 \hline
\end{tabular}

\end{table}
The set of singularities $\mathbf{S}=2\mathbf{A_6}\oplus2\mathbf{A}_3$ marked as exceptional is one of the special cases listed in the statement. We have $|E^+(\mathbf{T})|=4$ and the group $O(\mathbf{S})$ is generated by
\begin{itemize}
  \item a nontrivial symmetry of $\mathbf{A}_3$, mapped to $(1,1)\in E^+(\mathbf{T})$, or
  \item the transposition $\mathbf{A}_3\leftrightarrow\mathbf{A}_3$, mapped to $(1,1)\in E^+(\mathbf{T})$, or
  \item a nontrivial symmetry of $\mathbf{A}_6$, mapped to $(-1,1)\in E^+(\mathbf{T})$, or
  \item the transposition $\mathbf{A}_6\leftrightarrow\mathbf{A}_6$, mapped to $(-1,1)\in E^+(\mathbf{T})$.
\end{itemize}
It follows that $d^{\bot}_+$ is not surjective.

Finally, what remains are the $3$ sets of singularities
\begin{align*}
     \mathbf{D}_4\oplus2\mathbf{A}_4\oplus3\mathbf{A}_2,\quad 2\mathbf{A}_7\oplus2\mathbf{A}_2,\quad\mathbf{A}_4\oplus2\mathbf{A}_3\oplus2\mathbf{A}_2
\end{align*}
to which Lemma \ref{2irr+} does not apply
and we need to compute the groups $E^+(\mathbf{T})$ directly from the definition (\ref{E+(N)}). For
\begin{align*}
  \mathbf{S}=2\mathbf{A}_7\oplus2\mathbf{A}_2,
\end{align*}
which is the last exceptional case listed in statement, we have $\Sigma_2^{\sharp}(\mathbf{T})\not\supset\Gamma_{2,2}$ and $|\cdot|_2$ needs to be modified by replacing $\chi_2$ with $\chi_2(u)= u \mod8\in\{1,3,5,7\}=\mathbb{Z}_2^\times/(\mathbb{Z}_2^\times)^2$ and we have to consider the full group $\Gamma_{2,0}$ instead of $\Gamma'_{2,0}$. The computation can be summarized as follows:
\begin{center}
\begin{tabular}{c|cc|ccc|}

 &\multicolumn{2}{|c|}{$\Gamma_{3,0}$}&\multicolumn{3}{|c|}{$\Gamma_{2,0}$}\\\hline
generator of $\Sigma_3^{\sharp}(T) $ & -1 & -1 & 1 & 1&1\\\hline
 $\Sigma_2^{\sharp}(T)=\{1\}$& 1 & 1 & 1 & 1&1\\
  \hline
   $\varphi(-1,-1) $ & -1 & -1 & -1 & -1&1\\
  \hline\cline{1-6}\cline{1-6}
&  \multicolumn{1}{|c|}{$\delta_3$} &  \multicolumn{1}{|c|}{$|\cdot|_3$} &  \multicolumn{1}{|c|}{$\delta_2$} &  \multicolumn{2}{|c|}{$|\cdot|_2$}
   \\\hline
 \multicolumn{1}{c|}{a symmetry of $\mathbf{A}_2$} &  \multicolumn{1}{|c|}{-1} &  \multicolumn{1}{|c|}{1} &  \multicolumn{1}{|c|}{1} & \multicolumn{1}{c|}{-1} &-1\\\hline
\multicolumn{1}{c|}{a transposition $\mathbf{A}_2\leftrightarrow \mathbf{A}_2$} &  \multicolumn{1}{|c|}{-1} &  \multicolumn{1}{|c|}{-1} &  \multicolumn{1}{c|}{1} & \multicolumn{1}{c|}{-1}&-1\\\hline
\multicolumn{1}{c|}{a symmetry of $\mathbf{A}_7$} &  \multicolumn{1}{|c|}{1} &  \multicolumn{1}{|c|}{1} &  \multicolumn{1}{|c|}{-1} & \multicolumn{1}{c|}{-1}&1\\\hline
\multicolumn{1}{c|}{a transposition $\mathbf{A}_7\leftrightarrow \mathbf{A}_7$} &  \multicolumn{1}{|c|}{1} &  \multicolumn{1}{|c|}{-1} &  \multicolumn{1}{c|}{-1} & \multicolumn{1}{c|}{-1}&1\\\hline

\end{tabular}
\end{center}
The rank of the matrix composed by the $7$ rows of the table (see Remark \ref{remark}) is $4<\operatorname{dim}\Gamma_{3,0}+\operatorname{dim}\Gamma_{2,0}$, which implies that $d^{\bot}_+$ is not surjective. For the other two cases, there are three irregular primes and the computation repeats literally that at the end of the proof of Proposition \ref{propM1Sconjconnected}; in both cases, the map $d^{\bot}_+$ turns out to be surjective.
\end{proof}

\bibliographystyle{amsplain}
\bibliography{mybibliography}

\end{document}

%% file: max_table.tex
\hbox to\hsize{\hss\vtop{\halign{$\singset{#}$\hss\quad&\hss$#$\hss\cr
\noalign{\hrule\vspace{2pt}}
\omit\strut Singularities\hss&(r,c)\cr
\noalign{\hrule\vspace{2pt}}
2E8 + A2 + A1 & (1, 0) \cr
E8 + E7 + A4 & (1, 0) \cr
E8 + E6 + D5 & (1, 0) \cr
E8 + E6 + A4 + A1 & (1, 0) \cr
E8 + D7 + 2A2 & (1, 0) \cr
E8 + A10 + A1 & (1, 0) \cr
E8 + A9 + A2 & (1, 0) \cr
E8 + A6 + A5 & (1, 0) \cr
E8 + A6 + A4 + A1 & (1, 0) \cr
E8 + A6 + A3 + A2 & (0, 1) \cr
E8 + 2A4 + A2 + A1 & (1, 0) \cr
E7 + E6 + A6 & (1, 0) \cr
E7 + A12 & (1, 1) \cr
E7 + A10 + A2 & (0, 1) \cr
E7 + A8 + A4 & (2, 0) \cr
E7 + 2A6 & (0, 1) \cr
E7 + A6 + A4 + A2 & (1, 0) \cr
2E6 + D7 & (1, 0) \cr
E6 + D13 & (1, 0) \cr
E6 + D9 + A4 & (1, 0) \cr
E6 + A13 & (1, 0) \cr
E6 + A12 + A1 & (1, 0) \cr
D15 + 2A2 & (1, 0) \cr
D11 + A6 + A2 & (0, 1) \cr
D9 + A6 + 2A2 & (1, 0) \cr
D7 + A10 + A2 & (0, 1) \cr
D7 + 2A6 & (0, 1) \cr
D7 + A6 + A4 + A2 & (0, 1) \cr
D7 + 2A4 + 2A2 & (1, 0) \cr
\crcr}}\qquad\qquad
\vtop{\halign{$\singset{#}$\hss\quad&\hss$#$\hss\cr
\noalign{\hrule\vspace{2pt}}
\omit\strut Singularities\hss&(r,c)\cr
\noalign{\hrule\vspace{2pt}}
A18 + A1 & (1, 1) \cr
A17 + A2 & (1, 1) \cr
A16 + A2 + A1 & (1, 0) \cr
A15 + 2A2 & (0, 1) \cr
A14 + A5 & (0, 2) \cr
A14 + A3 + A2 & (0, 2) \cr
A13 + A6 & (0, 2) \cr
A13 + A4 + A2 & (1, 0) \cr
A12 + A6 + A1 & (1, 1) \cr
A12 + A5 + A2 & (1, 1) \cr
A12 + A4 + A2 + A1 & (0, 1) \cr
A12 + A3 + 2A2 & (2, 0) \cr
A11 + A6 + A2 & (0, 2) \cr
A10 + A9 & (1, 1) \cr
A10 + A8 + A1 & (0, 1) \cr
A10 + A7 + A2 & (0, 2) \cr
A10 + A6 + A3 & (0, 2) \cr
A10 + A6 + A2 + A1 & (1, 0) \cr
A10 + A5 + A4 & (1, 0) \cr
A10 + A4 + A3 + A2 & (0, 1) \cr
A9 + A8 + A2 & (1, 1) \cr
A9 + A6 + 2A2 & (1, 0) \cr
A8 + A6 + A5 & (1, 1) \cr
A8 + A6 + A4 + A1 & (0, 1) \cr
A8 + A6 + A3 + A2 & (0, 3) \cr
A7 + 2A6 & (0, 2) \cr
A7 + A6 + A4 + A2 & (0, 1) \cr
2A6 + A5 + A2 & (2, 0) \cr
2A6 + A4 + A2 + A1 & (0, 1) \cr
A6 + 2A4 + A3 + A2 & (2, 0) \cr
\crcr}}\hss} 

%% file: table18.tex
\hbox to\hsize{\hss\vtop{\halign{$\singset{#}$\hss\quad&\hss$#$\hss\cr
E8 + D10 \cr
E8+ D9 + A1 \cr
2E7 +2A2 \cr
E7 + D11 \cr
E7 + D9 + A2 \cr
D18 \cr
D17 + A1 \cr
\crcr}}\qquad
\vtop{\halign{$\singset{#}$\hss\quad&\hss$#$\hss\cr
D14 + A4 \cr
D10 + A8 \cr
D10 + 2A4  \cr
2D9 \cr
D9 + A8 + A1 \cr
D6 + 3A4 \cr
\crcr}}\qquad
\vtop{\halign{$\singset{#}$\hss\quad&\hss$#$\hss\cr
2D5 + A8 \cr
2D5 + 2A4 \cr
D5 + A9 + A4 \cr
D5 + A8 + A5  \cr
D5 + A5 + 2A4 \cr
2A9 \cr
\crcr}}\hss} 